\documentclass{article}

\usepackage{shuffle,yfonts}
\DeclareFontFamily{U}{shuffle}{}
\DeclareFontShape{U}{shuffle}{m}{n}{ <-8>shuffle7 <8->shuffle10}{}
\newcommand{\sha}{\shuffle}

\usepackage{amsmath,amssymb,amsthm}
\usepackage{etex}
\reserveinserts{28}
\usepackage[all]{xy}
\usepackage{tikz}
\usetikzlibrary{decorations.pathreplacing}
\newcommand{\tikznode}[3][inner sep=0pt]{\tikz[remember
picture,baseline=(#2.base)]{\node(#2)[#1]{$#3$};}}

\allowdisplaybreaks
\theoremstyle{plain}
\newtheorem{thm}{Theorem}[section]
\newtheorem{lem}[thm]{Lemma}
\newtheorem{cor}[thm]{Corollary}

\newtheorem{pro}[thm]{Proposition}

\theoremstyle{definition}

\newtheorem{re}[thm]{Remark}
\newtheorem{exa}[thm]{Example}
\newtheorem{defn}{Definition}[section]

 \allowdisplaybreaks

\newcommand{\vsmall}{\rotatebox[origin=c]{-90}{$<$}}

\newcommand\ga{{\alpha}}
\newcommand\gb{{\beta}}
\newcommand\om{{\omega}}
\newcommand\bfu{{\boldsymbol{\sl{u}}}}
\newcommand\bfv{{\boldsymbol{\sl{v}}}}
\newcommand{\bfs}{{\boldsymbol{\sl{s}}}}
\newcommand\bfk{{\boldsymbol{\sl{k}}}}
\newcommand\bfl{{\boldsymbol{\sl{l}}}}
\newcommand\bfn{{\boldsymbol{\sl{n}}}}

\newcommand\bfeps{{\boldsymbol \varepsilon}}
\newcommand\bfsi{{\boldsymbol \sigma}}
\newcommand\gl{{\lambda}}
\newcommand\oldLi{{\Li}}

\newcommand\eps{{\varepsilon}}

\newcommand\xx{\mbox{\bfseries \itshape x}}
\def\R{\mathbb{R}}
\def\N{\mathbb{N}}

\def\Q{\mathbb{Q}}
\def\CC{\mathbb{C}}

\def\su{\sum\limits_{n=1}^\infty}

\def\z{\zeta}
\def\gz{\zeta}

\def\Li{{\rm Li}}

\def\tt{\left(\frac{1-t}{1+t} \right)}
\def\xx{\left(\frac{1-x}{1+x} \right)}

\def\A{{\rm A}}

\newcommand\SSYT{{\text{SSYT}}}
\usepackage{ytableau}

\begin{document}
\title {\bf Explicit Relations between Kaneko--Yamamoto Type Multiple Zeta Values and Related Variants}
\author{
{Ce Xu${}^{a,}$\thanks{Email: cexu2020@ahnu.edu.cn}\quad and Jianqiang Zhao${}^{b,}$\thanks{Email: zhaoj@ihes.fr}}\\[1mm]
\small a. School of Mathematics and Statistics, Anhui Normal University, Wuhu 241000, PRC\\
\small b. Department of Mathematics, The Bishop's School, La Jolla, CA 92037, USA \\
[5mm]
Dedicated to Professor Masanobu Kaneko on the occasion of his 60th birthday}

\date{}
\maketitle \noindent{\bf Abstract.} In this paper we first establish several integral identities. These integrals are of the form
\[\int_0^1 x^{an+b} f(x)\,dx\quad (a\in\{1,2\},\ b\in\{-1,-2\})\]
where $f(x)$ is a single-variable multiple polylogarithm function or $r$-variable multiple polylogarithm function or Kaneko--Tsumura A-function (this is a single-variable multiple polylogarithm function of level two). We find that these integrals can be expressed in terms of multiple zeta (star) values and their related variants (multiple $t$-values, multiple $T$-values, multiple $S$-values etc.), and multiple harmonic (star) sums and their related variants (multiple $T$-harmonic sums, multiple $S$-harmonic sums etc.), which are closely related to some special types of Schur multiple zeta values and their generalization. Using these integral identities, we prove many explicit evaluations of Kaneko--Yamamoto multiple zeta values and their related variants. Further, we derive some relations involving multiple zeta (star) values and their related variants.



\section{Introduction and Notations}
\subsection{Multiple zeta values (MZVs) and Schur MZVs}
We begin with some basic notations. A finite sequence $\bfk \equiv {\bfk_r}:= (k_1,\dotsc, k_r)$ of positive integers is called a \emph{composition}. We put
\[|\bfk|:=k_1+\dotsb+k_r,\quad {\rm dep}(\bfk):=r,\]
and call them the weight and the depth of $\bfk$, respectively.

For $0\leq j\leq i$, we adopt the following notations:
\begin{align*}
&{\bfk}_i^j:=(\underbrace{k_{i+1-j},k_{i+2-j},\dotsc,k_i}_{j\ \text{components}})
\end{align*}
and
\begin{align*}
&{\bfk}_i\equiv{\bfk}_i^i:=(k_1,k_2,\dotsc,k_i),
\end{align*}
where ${\bfk}_i^0:=\emptyset$\quad $(i\geq 0)$. If $i<j$, then ${\bfk}_i^j:=\emptyset$.

For a composition $\bfk_r=(k_1,\dotsc,k_r)$ and positive integer $n$, the multiple harmonic sums (MHSs) and multiple harmonic star sums (MHSSs) are defined by
\begin{align*}
\gz_n(k_1,\dotsc,k_r):=\sum\limits_{0<m_1<\cdots<m_r\leq n } \frac{1}{m_1^{k_1}m_2^{k_2}\cdots m_r^{k_r}}
\end{align*}
and
\begin{align*}
\gz^\star_n(k_1,\dotsc,k_r):=\sum\limits_{0<m_1\leq \cdots\leq m_r\leq n} \frac{1}{m_1^{k_1}m_2^{k_2}\cdots m_r^{k_r}},
\end{align*}
respectively. If $n<r$ then ${\gz_n}(\bfk_r):=0$ and ${\zeta _n}(\emptyset )={\gz^\star_n}(\emptyset ):=1$.
The multiple zeta values (abbr. MZVs) and the multiple zeta-star value (abbr. MZSVs) are defined by
\begin{equation*}
\zeta(\bfk):=\lim_{n\to\infty } \gz_n(\bfk)
\qquad\text{and}\qquad
\gz^\star(\bfk):=\lim_{n\to\infty } \gz_n^\star(\bfk),
\end{equation*}
respectively. These series converge if and only if $k_r\ge2$ so we call a composition $\bfk_r$ \emph{admissible} if this is the case.

The systematic study of multiple zeta values began in the early 1990s with the works of Hoffman \cite{H1992} and Zagier \cite{DZ1994}. Due to their surprising and sometimes mysterious appearance in the study of many branches of mathematics and theoretical physics, these special values have attracted a lot of attention and interest in the past three decades (for example, see the book by the second author \cite{Z2016}). A common generalization of the MZVs and MZSVs is given
by the Schur multiple zeta values \cite{MatsumotoNakasuji2020,NPY2018}, which are defined using the skew Young tableaux.
For example, for integers $a,c,d,e,f\geq 1,\,c,g \geq 2$ the following sum is an example of a Schur multiple mixed values
\begin{equation}\label{equ:SchurEg}
\zeta\left(\ {\ytableausetup{centertableaux, boxsize=1.2em}
		\begin{ytableau}
		\none &  a& b & c \\
		d& e & \none\\
		f & g & \none
		\end{ytableau}}\ \right)
	= \sum_{{\scriptsize
			\arraycolsep=1.4pt\def\arraystretch{0.8}
			\begin{array}{ccccccc}
			&&m_a&\leq&m_b&\leq& m_c \\
			&&\vsmall&& &&  \\
			m_d&\leq&m_e&&&& \\
			\vsmall&&\vsmall&&& \\
			m_f&\leq&m_g&&&&
			\end{array} }} \frac{1}{m_a^{\,a} \,\, m_b^b \,\,  m_c^c \,\,  m_d^d \,\,  m_e^e \,\,  m_f^f}  \,,
\end{equation}

In this paper, we shall study some families of variations of MZVs. First, consider the following special form of
the Schur multiple zeta values.
\begin{defn} (cf. \cite{KY2018})
For any two compositions of positive integers $\bfk=(k_1,\dotsc,k_r)$ and $\bfl=(l_1,\dotsc,l_s)$, define
\begin{align}\label{equ:KYMZVs}
\zeta(\bfk\circledast{\bfl^\star})
:=&\sum\limits_{0<m_1<\cdots<m_r=n_s\geq \cdots \geq n_1 > 0} \frac{1}{m_1^{k_1}\cdots m_r^{k_r}n_1^{l_1}\cdots n_s^{l_s}} \\
=& \su \frac{\gz_{n-1}(k_1,\dotsc,k_{r-1})\gz^\star_n(l_1,\dotsc,l_{s-1})}{n^{k_r+l_s}}. \notag
\end{align}
We call them \emph{Kaneko--Yamamoto multiple zeta values} (K--Y MZVs for short).
\end{defn}

The K--Y MZVs defined by \eqref{equ:KYMZVs} are the special cases of the Schur multiple zeta values
$\zeta_\gl(\bfs)$ given by the following skew Young tableaux of anti-hook type
\[
\bfs={\footnotesize \ytableausetup{centertableaux, boxsize=1.8em}
\begin{ytableau}
       \none & \none & \none &  \tikznode{a1}{\scriptstyle k_1} \\
       \none & \none & \none & \vdots \\
       \none & \none & \none & \scriptstyle k_{r-1} \\
       \tikznode{a2}{\scriptstyle l_1} & \cdots & \scriptstyle l_{s-1} &   \tikznode{a3}{\scriptstyle x}
\end{ytableau}}
\]
where $x=k_r+l_s$ and $\gl$ is simply the Young diagram underlying the above tableaux.

\subsection{Variations of MZVs with even/odd summation indices}
One may modify the definition MZVs by restricting the summation indices to even/odd numbers.
These values are apparently NOT in the class of Schur multiple zeta values.
For instance, in recent papers \cite{KanekoTs2018b,KanekoTs2019}, Kaneko and Tsumura introduced a new kind of
multiple zeta values of level two, called \emph{multiple T-values} (MTVs), defined
for admissible compositions $\bfk=(k_1,\dotsc,k_r)$ by
\begin{align}
T(\bfk):&=2^r \sum_{0<m_1<\cdots<m_r\atop m_i\equiv i\ {\rm mod}\ 2} \frac{1}{m_1^{k_1}m_2^{k_2}\cdots m_r^{k_r}}\nonumber\\
&=2^r\sum\limits_{0<n_1<\cdots<n_r} \frac{1}{(2n_1-1)^{k_1}(2n_2-2)^{k_2}\cdots (2n_r-r)^{k_r}}.
\end{align}
This is in contrast to Hoffman's \emph{multiple $t$-values} (MtVs) defined in \cite{H2019} as follows: for
admissible compositions $\bfk=(k_1,\dotsc,k_r)$
\begin{align*}
t(\bfk):=\sum_{0<m_1<\cdots<m_r\atop \forall m_i: odd} \frac{1}{m_1^{k_1}m_2^{k_2}\cdots m_r^{k_r}}
=\sum\limits_{0<n_1<\cdots<n_r} \frac{1}{(2n_1-1)^{k_1}(2n_2-1)^{k_2}\cdots (2n_r-1)^{k_r}}.
\end{align*}
Moreover, in \cite{H2019} Hoffman also defined its star version, called \emph{multiple $t$-star value} (MtSVs), as follows:
\begin{align*}
t^\star(\bfk):=\sum_{0<m_1\leq \cdots\leq m_r\atop \forall m_i: odd} \frac{1}{m_1^{k_1}m_2^{k_2}\cdots m_r^{k_r}}
=\sum\limits_{0<n_1\leq \cdots\leq n_r} \frac{1}{(2n_1-1)^{k_1}(2n_2-1)^{k_2}\cdots (2n_r-1)^{k_r}}.
\end{align*}
Very recently, the authors have defined another variant of multiple zeta values in \cite{XZ2020}, called \emph{multiple mixed values} or \emph{multiple $M$-values} (MMVs for short). For  $\bfeps=(\varepsilon_1, \dots, \varepsilon_r)\in\{\pm 1\}^r$
and admissible compositions $\bfk=(k_1,\dotsc,k_r)$,
\begin{align}
M(\bfk;\bfeps):&=\sum_{0<m_1<\cdots<m_r} \frac{(1+\varepsilon_1(-1)^{m_1}) \cdots (1+\varepsilon_r(-1)^{m_r})}{m_1^{k_1} \cdots m_r^{k_r}}\nonumber\\
&=\sum_{0<m_1<\cdots<m_r\atop 2| m_j\ \text{if}\ \varepsilon_j=1\ \text{and}\ 2\nmid m_j\ \text{if}\ \varepsilon_j=-1} \frac{2^r}{m_1^{k_1}m_2^{k_2} \cdots m_r^{k_r}}.
\end{align}
For brevity, we put a check on top of the component $k_j$ if $\varepsilon_j=-1$. For example,
\begin{align*}
M(1,2,\check{3})=&\, \sum_{0<m_1<m_2<m_3} \frac{(1+(-1)^{m_1}) (1+(-1)^{m_2}) (1-(-1)^{m_3})}{m_1 m_2^{2}m_3^{3}}\\
=&\, \sum_{0<l<m<n} \frac{8}{(2\ell)  (2m)^{2} (2n-1)^{3}}.
\end{align*}
It is obvious that MtVs satisfy the series stuffle relation, however, it is nontrivial to see that MTVs can be expressed using iterated integral and satisfy both the duality relations (see \cite[Theorem 3.1]{KanekoTs2019}) and the integral shuffle relations (see \cite[Theorem 2.1]{KanekoTs2019}). Similar to MZVs, MMVs satisfy both series stuffle relations and the integral shuffle relations. Moreover, in \cite{XZ2020}, we also introduced and studied
a class of MMVs that is opposite to MTVs, called \emph{multiple S-values} (MSVs). For admissible compositions $\bfk=(k_1,\dotsc,k_r)$,
\begin{align}
S(\bfk):&=2^r \sum_{0<m_1<\cdots<m_r\atop m_i\equiv i-1\ {\rm mod}\ 2} \frac{1}{m_1^{k_1}m_2^{k_2}\cdots m_r^{k_r}}\nonumber\\
&=2^r \sum_{0<n_1<\cdots<n_r} \frac{1}{(2n_1)^{k_1}(2n_2-1)^{k_2}\cdots (2n_r-r+1)^{k_r}}.
\end{align}
It is clear that every MMV can be written as a linear combination of alternating MZVs
(also referred to as Euler sums or colored multiple zeta values) defined as follows.
For $\bfk\in\N^r$ and $\bfeps\in\{\pm 1\}^r$, if $(k_r,\eps_r)\ne(1,1)$ (called \emph{admissible} case) then
\begin{equation*}
 \zeta(\bfk;\bfeps):=\sum\limits_{0<m_1<\cdots<m_r} \frac{\eps_1^{m_1}\cdots \eps_r^{m_r} }{ m_1^{k_1}\cdots m_r^{k_r}}.
\end{equation*}
We may compactly indicate the presence of an alternating sign as follows.
Whenever $\eps_j=-1$,  we place a bar over the corresponding integer exponent $k_j$. For example,
\begin{equation*}
\zeta(\bar 2,3,\bar 1,4)=\zeta(  2,3,1,4;-1,1,-1,1).
\end{equation*}
Similarly, a star-version of alternating MZVs (called \emph{alternating multiple zeta-star values}) is defined by
\begin{equation*}
 \gz^\star(\bfk;\bfeps):=\sum\limits_{0<m_1\leq\cdots\leq m_r}  \frac{\eps_1^{m_1}\cdots \eps_r^{m_r} }{ m_1^{k_1}\cdots m_r^{k_r}}.
\end{equation*}
Deligne showed that the rational spaced generated by alternating MZVs of weight $w$ is bounded by the Fibonacci number $F_w$
where $F_0=F_1=1$. In \cite[Theorem 7.1]{XZ2020} we showed that the rational spaced generated by MMVs of weight $w$
is bounded by $F_w-1$. The missing piece is the one-dimensional space generated by $\ln^w 2$.

\subsection{Variations of Kaneko--Yamamoto MZVs with even/odd summation indices}
Now, we introduce the $T$-variant of Kaneko--Yamamoto MZVs.
For positive integers $m$ and $n$ such that $n\ge m$, we define
\begin{align*}
&D_{n,m} :=
\left\{
  \begin{array}{ll}
\Big\{(n_1,n_2,\dotsc,n_m)\in\N^{m} \mid 0<n_1\leq n_2< n_3\leq \cdots \leq n_{m-1}<n_{m}\leq n \Big\},\phantom{\frac12}\ & \hbox{if $2\nmid m$;} \\
\Big\{(n_1,n_2,\dotsc,n_m)\in\N^{m} \mid 0<n_1\leq n_2< n_3\leq \cdots <n_{m-1}\leq n_{m}<n \Big\},\phantom{\frac12}\ & \hbox{if $2\mid m$,}
  \end{array}
\right.  \\
&E_{n,m} :=
\left\{
  \begin{array}{ll}
\Big\{(n_1,n_2,\dotsc,n_{m})\in\N^{m}\mid 1\leq n_1<n_2\leq n_3< \cdots< n_{m-1}\leq n_{m}< n \Big\},\phantom{\frac12}\ & \hbox{if $2\nmid m$;} \\
\Big\{(n_1,n_2,\dotsc,n_{m})\in\N^{m}\mid 1\leq n_1<n_2\leq n_3< \cdots \leq n_{m-1}< n_{m}\leq n \Big\}, \phantom{\frac12}\ & \hbox{if $2\mid m$.}
  \end{array}
\right.
\end{align*}

\begin{defn} (\cite[Defn. 1.1]{XZ2020}) For positive integer $m$, define
\begin{align}
&T_n({\bfk_{2m-1}}):= \sum_{\bfn\in D_{n,2m-1}} \frac{2^{2m-1}}{(\prod_{j=1}^{m-1} (2n_{2j-1}-1)^{k_{2j-1}}(2n_{2j})^{k_{2j}})(2n_{2m-1}-1)^{k_{2m-1}}},\label{MOT}\\
&T_n({\bfk_{2m}}):= \sum_{\bfn\in D_{n,2m}} \frac{2^{2m}}{\prod_{j=1}^{m} (2n_{2j-1}-1)^{k_{2j-1}}(2n_{2j})^{k_{2j}}},\label{MET}\\
&S_n({\bfk_{2m-1}}):= \sum_{\bfn\in E_{n,2m-1}} \frac{2^{2m-1}}{(\prod_{j=1}^{m-1} (2n_{2j-1})^{k_{2j-1}}(2n_{2j}-1)^{k_{2j}})(2n_{2m-1})^{k_{2m-1}}},\label{MOS}\\
&S_n({\bfk_{2m}}):= \sum_{\bfn\in E_{n,2m}} \frac{2^{2m}}{\prod_{j=1}^{m} (2n_{2j-1})^{k_{2j-1}}(2n_{2j}-1)^{k_{2j}}},\label{MES}
\end{align}
where $T_n({\bfk_{2m-1}}):=0$ if $n<m$, and $T_n({\bfk_{2m}})=S_n({\bfk_{2m-1}})=S_n({\bfk_{2m}}):=0$ if $n\leq m$. Moreover, for convenience sake, we set $T_n(\emptyset)=S_n(\emptyset):=1$. We call \eqref{MOT} and \eqref{MET} \emph{multiple $T$-harmonic sums} ({\rm MTHSs} for short), and call \eqref{MOS} and \eqref{MES}  \emph{multiple $S$-harmonic sums} ({\rm MSHSs} for short).
\end{defn}
In \cite{XZ2020}, we used the MTHSs and MSHSs to define the convoluted $T$-values $T({\bfk}\circledast {\bfl})$, which can be regarded as a $S$- or $T$-variant of K--Y MZVs.

\begin{defn} (\cite[Defn. 1.2]{XZ2020}) For positive integers $m$ and $p$, the \emph{convoluted $T$-values} are defined by
\begin{align}\label{equ:schur1}
T({\bfk_{2m}}\circledast{\bfl_{2p}})=&\,2\su \frac{T_n({\bfk_{2m-1}})T_n({\bfl_{2p-1}})}{(2n)^{k_{2m}+l_{2p}}},\\
T({\bfk_{2m-1}}\circledast{\bfl_{2p-1}})=&\,2\su \frac{T_n({\bfk_{2m-2}})T_n({\bfl_{2p-2}})}{(2n-1)^{k_{2m-1}+l_{2p-1}}},\\
T({\bfk_{2m}}\circledast{\bfl_{2p-1}})=&\,2\su \frac{T_n({\bfk_{2m-1}})S_n({\bfl_{2p-2}})}{(2n)^{k_{2m}+l_{2p-1}}},\\
T({\bfk_{2m-1}}\circledast{\bfl_{2p}})=&\,2\su \frac{T_n({\bfk_{2m-2}})S_n({\bfl_{2p-1}})}{(2n-1)^{k_{2m-1}+l_{2p}}}.\label{equ:schur4}
\end{align}
We may further define the \emph{convoluted $S$-values} by
\begin{align}\label{equ:schur1}
S({\bfk_{2m}}\circledast{\bfl_{2p}})=&\,2\su \frac{S_n({\bfk_{2m-1}})S_n({\bfl_{2p-1}})}{(2n-1)^{k_{2m}+l_{2p}}},\\
S({\bfk_{2m-1}}\circledast{\bfl_{2p-1}})=&\,2\su \frac{S_n({\bfk_{2m-2}})S_n({\bfl_{2p-2}})}{(2n)^{k_{2m-1}+l_{2p-1}}}.\label{equ:schur6}
\end{align}
\end{defn}
In view of the interpretation of K--Y MZVs as special Schur MZVs, one may wonder if
Schur MZVs can be generalized so that the convoluted $S$- and $T$-values are special cases.

\subsection{Schur MZVs modulo $N$}
We now generalize the concept of Schur multiple zeta functions (resp. values)
to Schur multiple zeta functions (resp. values) modulo any positive integer $N$,
the case $N=2$ of which contain all the MMVs as special cases.

It turns out that when $N=2$ the only differences between these values and the Schur MZVs is
that each box in the Young diagram is decorated by either ``0'' or ``1'' at upper left corner so that the running
index appearing in that box must be either even or odd. For example, a variation of
the example in \eqref{equ:SchurEg} can be given as follows:
\begin{equation*}
\zeta\left(\ {\ytableausetup{centertableaux, boxsize=1.2em}
		\begin{ytableau}
		\none & {}^{\text{0}} a& {}^{\text{0}}b & {}^{\text{1}}c \\
		{}^{\text{1}}d& {}^{\text{1}}e & \none\\
		{}^{\text{0}}f & {}^{\text{1}}g & \none
		\end{ytableau}}\ \right)
	:= \sum_{{\scriptsize
			\arraycolsep=1.4pt\def\arraystretch{0.8}
			\begin{array}{cccccccl}
			&&m_a&\leq&m_b&\leq& m_c \quad &\ 2|m_a,2|m_b,2\nmid m_c\\
			&&\vsmall&& &&&\  \\
			m_d&\leq&m_e&&&& &\ 2\nmid m_d,2\nmid m_e\\
			\vsmall&&\vsmall&&&& \\
			m_f&\leq&m_g&&&& &\ 2|m_f,2\nmid m_g
			\end{array} }} \frac{2^7}{m_a^{\,a} \,\, m_b^b \,\,  m_c^c \,\,  m_d^d \,\,  m_e^e \,\,  m_f^f}  \,,
\end{equation*}

We now briefly describe this idea in general. For a skew Young diagram
$\gl$ with $n$ boxes (denoted by $\sharp(\gl)=n$),
let $T(\gl, X)$ be the set of all Young tableaux of shape $\gl$ over a set $X$.
Let $D(\gl)=\{(i,j): 1\le i\le r, \ga_i\le j\le \gb_i\}=\{(i,j): 1\le j\le s, a_j\le i\le b_j\}$ be the skew Young diagram of $\gl$
so that $(i,j)$ refers to the box on the $i$th row and $j$th column of $\gl$.
Fix any positive integer $N$, we may decorate $D(\gl)$ by putting a residue class $\pi_{ij}$ mod $N$ at
the upper left corner of $(i,j)$-th box. We call such a decorated diagram a Young diagram modulo $N$,
denoted by $\gl^\pi$. Further, we define the set of semi-standard skew Young tableaux of shape $\gl^\pi$ by
$$
\SSYT(\gl^\pi):=\left\{(m_{i,j})\in T(\gl, \N)\left|
 \aligned
 & m_{i,\ga_i}\le m_{i,\ga_i+1}\le  \dotsm \le m_{i,\gb_i},\ \  m_{a_j,j}< m_{a_j+1,j}<\dotsm<m_{b_j,j},\\
 & m_{i,j}\equiv \pi_{i,j} \pmod{N} \ \ \forall 1\le i\le r, \ga_i\le j\le \gb_i
 \endaligned
\right.\right\}.
$$
For $\bfs = (s_{i,j} )\in T(\gl,\CC)$, the \emph{Schur multiple zeta function}   \emph{modulo} $N$
associated with $\gl^\pi$ is defined by the series
\begin{equation*}
\zeta_{\gl^\pi}(\bfs):=\sum_{M\in \SSYT(\gl^\pi)} \frac{2^{\sharp(\gl)} }{M^\bfs}
\end{equation*}
where $M^\bfs=(m_{i,j})^\bfs:=\prod{}_{(i,j)\in D(\gl)} m_{i,j}^{s_{i,j}}$.
Similar to \cite[Lemma 2.1]{NPY2018}, it is not too hard to prove that the above series converges
absolutely whenever $\bfs\in W_\gl$ where
$$
W_\gl:=\left\{\bfs=(s_{i,j})\in T(\gl, \CC) \left|
 \aligned
 & \Re(s_{i,j})\ge  1 \ \forall (i,j)\in D(\gl)\setminus  C(\gl) \\
 &  \Re(s_{i,j})> 1 \ \forall (i,j)\in C(\gl)
 \endaligned
\right.\right\},
$$
where $C(\gl)$ is the set of all corners of $\gl$. But this domain of convergence is not ideal. To define
the most accurate domain of convergence we need the following terminology. Given any two boxes $B_1$ and $B_2$
in $\gl$ we define $B_1\preceq B_2$ if $B_1$ is to the left or above $B_2$, namely, $B_1$
must be in the gray area in the picture
\begin{tikzpicture}[scale=0.05]
\filldraw [gray!30!white] (0,1) -- (3,1) -- (3,4) -- (8,4) -- (8,6) -- (0,6) -- (0,1);
\fill[black] (3,3) -- (4,3)  -- (4,4)  -- (3,4)  -- (3,3);
\end{tikzpicture}
where the $B_2$ is the black box.

An \emph{allowable move} along a path from box $B$ is a move to a box $C$ such that $B\preceq C$ and all boxes above $C$ and to the left of $C$, if there are any, are already covered by the previous moves along the path.
An \emph{allowable path} in a skew Young diagram is a sequence of allowable moves covering all
the boxes without backtracking.
Then the  domain of convergence of $M_{\gl^\pi}(\bfs)$ is the subset of $W_\gl$
defined by the condition that $\Re( \sum_{s_{ij}\in\mathcal{P}_\ell} s_{ij})>\ell$ for all allowable paths
$\mathcal{P}$, where $\mathcal{P}_\ell$ is the sub-path of $\mathcal{P}$ covering the last $\ell$ boxes ending at a corner.
For example, the graph 
$
{ \ytableausetup{centertableaux, boxsize=.5em}\begin{ytableau}
       \none & \none & \scriptscriptstyle 1 &  \scriptscriptstyle 4 &   \scriptscriptstyle 6 \\
        \scriptscriptstyle 2 &  \scriptscriptstyle 5 & \scriptscriptstyle 7 & \none &  \none \\
        \scriptscriptstyle 3 & \none & \none & \none & \none \\
      \end{ytableau}}
$
($1\to 2\to \cdots \to 7$) shows an allowable path in a skew Young diagram.

Similar to K--Y MZVs, the above convoluted $S$- and $T$-values are all special cases of Schur MZVs modulo 2 corresponding to anti-hook type Young diagrams. The six convoluted $S$- or $T$-values in
\eqref{equ:schur1}-\eqref{equ:schur6} are all given by  mod 2 Schur MZVs
$\zeta_{\gl_j^{\pi_j}}$ ($1\le j\le 6$) below, respectively:
$$
\aligned
& \gl_1^{\pi_1}={ \ytableausetup{centertableaux, boxsize=1.8em}\begin{ytableau}
       \none & \none & \none & \none &  \tikznode{a1}{~} \\
       \none & \none & \none & \none &  \tikznode{a2}{~} \\
       \none & \none & \none & \none & \vdots \\
       \none & \none & \none & \none & \tikznode{a3}{~} \\
      \tikznode{a5}{~} &\tikznode{a4}{~}& \cdots & \tikznode{a7}{~} & \tikznode{a6}{~}
\end{ytableau}},\qquad
\gl_2^{\pi_2}={ \ytableausetup{centertableaux, boxsize=1.8em}\begin{ytableau}
       \none & \none & \none & \none &  \tikznode{b1}{~} \\
       \none & \none & \none & \none &  \tikznode{b2}{~} \\
       \none & \none & \none & \none & \vdots \\
       \none & \none & \none & \none & \tikznode{b8}{~} \\
      \tikznode{b5}{~} &\tikznode{b4}{~}& \cdots & \tikznode{b6}{~} & \tikznode{b7}{~}
\end{ytableau}},\qquad
\gl_3^{\pi_3}={ \ytableausetup{centertableaux, boxsize=1.8em}\begin{ytableau}
       \none & \none & \none & \none &  \tikznode{c1}{~} \\
       \none & \none & \none & \none &  \tikznode{c2}{~} \\
       \none & \none & \none & \none & \vdots \\
       \none & \none & \none & \none & \tikznode{c3}{~} \\
     \tikznode{c4}{~} &\tikznode{c5}{~} & \cdots & \tikznode{c7}{~}&  \tikznode{c6}{~}
\end{ytableau}},\\
&\gl_4^{\pi_4}={ \ytableausetup{centertableaux, boxsize=1.8em}\begin{ytableau}
       \none & \none & \none & \none &  \tikznode{d1}{~} \\
       \none & \none & \none & \none &  \tikznode{d2}{~} \\
       \none & \none & \none & \none & \vdots \\
       \none & \none & \none & \none & \tikznode{d8}{~} \\
      \tikznode{d4}{~} &\tikznode{d5}{~} & \cdots &\tikznode{d6}{~} & \tikznode{d7}{~}
\end{ytableau}},\qquad
\gl_5^{\pi_5}={ \ytableausetup{centertableaux, boxsize=1.8em}\begin{ytableau}
       \none & \none & \none & \none &  \tikznode{e1}{~} \\
       \none & \none & \none & \none &  \tikznode{e2}{~} \\
       \none & \none & \none & \none & \vdots \\
       \none & \none & \none & \none & \tikznode{e3}{~} \\
      \tikznode{e5}{~} &\tikznode{e4}{~}& \cdots & \tikznode{e7}{~} & \tikznode{e6}{~}
\end{ytableau}},\qquad
\gl_6^{\pi_6}={ \ytableausetup{centertableaux, boxsize=1.8em}\begin{ytableau}
       \none & \none & \none & \none &  \tikznode{f1}{~} \\
       \none & \none & \none & \none &  \tikznode{f2}{~} \\
       \none & \none & \none & \none & \vdots \\
       \none & \none & \none & \none & \tikznode{f8}{~} \\
     \tikznode{f4}{~} &\tikznode{f5}{~} & \cdots & \tikznode{f6}{~}&  \tikznode{f7}{~}
\end{ytableau}}
\endaligned
$$
\tikz[overlay,remember picture]{%
\draw[decorate,decoration={brace},thick] ([yshift=0.5em,xshift=-0.5em]a1.north west) --
([yshift=0.5em,xshift=-0.5em]a1.north west) node[midway]{${}^\text{1}$};
\draw[decorate,decoration={brace},thick] ([yshift=0mm,xshift=0mm]a1.center) --
([yshift=0mm,xshift=0mm]a1.center) node[midway]{$k_1$};
\draw[decorate,decoration={brace},thick] ([yshift=0.5em,xshift=-0.5em]a2.north west) --
([yshift=0.5em,xshift=-0.5em]a2.north west) node[midway]{${}^\text{0}$};
\draw[decorate,decoration={brace},thick] ([yshift=0mm,xshift=0mm]a2.center) --
([yshift=0mm,xshift=0mm]a2.center) node[midway]{$k_2$};
\draw[decorate,decoration={brace},thick] ([yshift=0.5em,xshift=-0.5em]a3.north west) --
([yshift=0.5em,xshift=-0.5em]a3.north west) node[midway]{${}^\text{1}$};
\draw[decorate,decoration={brace},thick] ([yshift=0mm,xshift=0mm]a3.center) --
([yshift=0mm,xshift=0mm]a3.center) node[midway]{$k_{m'}$};
\draw[decorate,decoration={brace},thick] ([yshift=0.5em,xshift=-0.5em]a4.north west) --
([yshift=0.5em,xshift=-0.5em]a4.north west) node[midway]{${}^\text{0}$};
\draw[decorate,decoration={brace},thick] ([yshift=0mm,xshift=0mm]a4.center) --
([yshift=0mm,xshift=0mm]a4.center) node[midway]{$l_2$};
\draw[decorate,decoration={brace},thick] ([yshift=0.5em,xshift=-0.5em]a5.north west) --
([yshift=0.5em,xshift=-0.5em]a5.north west) node[midway]{${}^\text{1}$};
\draw[decorate,decoration={brace},thick] ([yshift=0mm,xshift=0mm]a5.center) --
([yshift=0mm,xshift=0mm]a5.center) node[midway]{$l_1$};
\draw[decorate,decoration={brace},thick] ([yshift=0.5em,xshift=-0.5em]a6.north west) --
([yshift=0.5em,xshift=-0.5em]a6.north west) node[midway]{${}^\text{0}$};
\draw[decorate,decoration={brace},thick] ([yshift=0mm,xshift=0mm]a6.center) --
([yshift=0mm,xshift=0mm]a6.center) node[midway]{$x_1$};
\draw[decorate,decoration={brace},thick] ([yshift=0.5em,xshift=-0.5em]a7.north west) --
([yshift=0.5em,xshift=-0.5em]a7.north west) node[midway]{${}^\text{1}$};
\draw[decorate,decoration={brace},thick] ([yshift=0mm,xshift=0mm]a7.center) --
([yshift=0mm,xshift=0mm]a7.center) node[midway]{$l_{p'}$};
\draw[decorate,decoration={brace},thick] ([yshift=0.5em,xshift=-0.5em]b1.north west) --
([yshift=0.5em,xshift=-0.5em]b1.north west) node[midway]{${}^\text{1}$};
\draw[decorate,decoration={brace},thick] ([yshift=0mm,xshift=0mm]b1.center) --
([yshift=0mm,xshift=0mm]b1.center) node[midway]{$k_1$};
\draw[decorate,decoration={brace},thick] ([yshift=0.5em,xshift=-0.5em]b2.north west) --
([yshift=0.5em,xshift=-0.5em]b2.north west) node[midway]{${}^\text{0}$};
\draw[decorate,decoration={brace},thick] ([yshift=0mm,xshift=0mm]b2.center) --
([yshift=0mm,xshift=0mm]b2.center) node[midway]{$k_2$};
\draw[decorate,decoration={brace},thick] ([yshift=0.5em,xshift=-0.5em]b8.north west) --
([yshift=0.5em,xshift=-0.5em]b8.north west) node[midway]{${}^\text{0}$};
\draw[decorate,decoration={brace},thick] ([yshift=0mm,xshift=0mm]b8.center) --
([yshift=0mm,xshift=0mm]b8.center) node[midway]{$k_{m''}$};
\draw[decorate,decoration={brace},thick] ([yshift=0.5em,xshift=-0.5em]b4.north west) --
([yshift=0.5em,xshift=-0.5em]b4.north west) node[midway]{${}^\text{0}$};
\draw[decorate,decoration={brace},thick] ([yshift=0mm,xshift=0mm]b4.center) --
([yshift=0mm,xshift=0mm]b4.center) node[midway]{$l_2$};
\draw[decorate,decoration={brace},thick] ([yshift=0.5em,xshift=-0.5em]b5.north west) --
([yshift=0.5em,xshift=-0.5em]b5.north west) node[midway]{${}^\text{1}$};
\draw[decorate,decoration={brace},thick] ([yshift=0mm,xshift=0mm]b5.center) --
([yshift=0mm,xshift=0mm]b5.center) node[midway]{$l_1$};
\draw[decorate,decoration={brace},thick] ([yshift=0.5em,xshift=-0.5em]b6.north west) --
([yshift=0.5em,xshift=-0.5em]b6.north west) node[midway]{${}^\text{0}$};
\draw[decorate,decoration={brace},thick] ([yshift=0mm,xshift=0mm]b6.center) --
([yshift=0mm,xshift=0mm]b6.center) node[midway]{$l_{p''}$};
\draw[decorate,decoration={brace},thick] ([yshift=0.5em,xshift=-0.5em]b7.north west) --
([yshift=0.5em,xshift=-0.5em]b7.north west) node[midway]{${}^\text{1}$};
\draw[decorate,decoration={brace},thick] ([yshift=0mm,xshift=0mm]b7.center) --
([yshift=0mm,xshift=0mm]b7.center) node[midway]{$x_2$};
\draw[decorate,decoration={brace},thick] ([yshift=0.5em,xshift=-0.5em]c1.north west) --
([yshift=0.5em,xshift=-0.5em]c1.north west) node[midway]{${}^\text{1}$};
\draw[decorate,decoration={brace},thick] ([yshift=0mm,xshift=0mm]c1.center) --
([yshift=0mm,xshift=0mm]c1.center) node[midway]{$k_1$};
\draw[decorate,decoration={brace},thick] ([yshift=0.5em,xshift=-0.5em]c2.north west) --
([yshift=0.5em,xshift=-0.5em]c2.north west) node[midway]{${}^\text{0}$};
\draw[decorate,decoration={brace},thick] ([yshift=0mm,xshift=0mm]c2.center) --
([yshift=0mm,xshift=0mm]c2.center) node[midway]{$k_2$};
\draw[decorate,decoration={brace},thick] ([yshift=0.5em,xshift=-0.5em]c3.north west) --
([yshift=0.5em,xshift=-0.5em]c3.north west) node[midway]{${}^\text{1}$};
\draw[decorate,decoration={brace},thick] ([yshift=0mm,xshift=0mm]c3.center) --
([yshift=0mm,xshift=0mm]c3.center) node[midway]{$k_{m'}$};
\draw[decorate,decoration={brace},thick] ([yshift=0.5em,xshift=-0.5em]c4.north west) --
([yshift=0.5em,xshift=-0.5em]c4.north west) node[midway]{${}^\text{0}$};
\draw[decorate,decoration={brace},thick] ([yshift=0mm,xshift=0mm]c4.center) --
([yshift=0mm,xshift=0mm]c4.center) node[midway]{$l_1$};
\draw[decorate,decoration={brace},thick] ([yshift=0.5em,xshift=-0.5em]c5.north west) --
([yshift=0.5em,xshift=-0.5em]c5.north west) node[midway]{${}^\text{1}$};
\draw[decorate,decoration={brace},thick] ([yshift=0mm,xshift=0mm]c5.center) --
([yshift=0mm,xshift=0mm]c5.center) node[midway]{$l_2$};
\draw[decorate,decoration={brace},thick] ([yshift=0.5em,xshift=-0.5em]c6.north west) --
([yshift=0.5em,xshift=-0.5em]c6.north west) node[midway]{${}^\text{0}$};
\draw[decorate,decoration={brace},thick] ([yshift=0mm,xshift=0mm]c6.center) --
([yshift=0mm,xshift=0mm]c6.center) node[midway]{$x_3$};
\draw[decorate,decoration={brace},thick] ([yshift=0.5em,xshift=-0.5em]c7.north west) --
([yshift=0.5em,xshift=-0.5em]c7.north west) node[midway]{${}^\text{1}$};
\draw[decorate,decoration={brace},thick] ([yshift=0mm,xshift=0mm]c7.center) --
([yshift=0mm,xshift=0mm]c7.center) node[midway]{$l_{p''}$};
\draw[decorate,decoration={brace},thick] ([yshift=0.5em,xshift=-0.5em]d1.north west) --
([yshift=0.5em,xshift=-0.5em]d1.north west) node[midway]{${}^\text{1}$};
\draw[decorate,decoration={brace},thick] ([yshift=0mm,xshift=0mm]d1.center) --
([yshift=0mm,xshift=0mm]d1.center) node[midway]{$k_1$};
\draw[decorate,decoration={brace},thick] ([yshift=0.5em,xshift=-0.5em]d2.north west) --
([yshift=0.5em,xshift=-0.5em]d2.north west) node[midway]{${}^\text{0}$};
\draw[decorate,decoration={brace},thick] ([yshift=0mm,xshift=0mm]d2.center) --
([yshift=0mm,xshift=0mm]d2.center) node[midway]{$k_2$};
\draw[decorate,decoration={brace},thick] ([yshift=0.5em,xshift=-0.5em]d8.north west) --
([yshift=0.5em,xshift=-0.5em]d8.north west) node[midway]{${}^\text{0}$};
\draw[decorate,decoration={brace},thick] ([yshift=0mm,xshift=0mm]d8.center) --
([yshift=0mm,xshift=0mm]d8.center) node[midway]{$k_{m''}$};
\draw[decorate,decoration={brace},thick] ([yshift=0.5em,xshift=-0.5em]d4.north west) --
([yshift=0.5em,xshift=-0.5em]d4.north west) node[midway]{${}^\text{0}$};
\draw[decorate,decoration={brace},thick] ([yshift=0mm,xshift=0mm]d4.center) --
([yshift=0mm,xshift=0mm]d4.center) node[midway]{$l_1$};
\draw[decorate,decoration={brace},thick] ([yshift=0.5em,xshift=-0.5em]d5.north west) --
([yshift=0.5em,xshift=-0.5em]d5.north west) node[midway]{${}^\text{1}$};
\draw[decorate,decoration={brace},thick] ([yshift=0mm,xshift=0mm]d5.center) --
([yshift=0mm,xshift=0mm]d5.center) node[midway]{$l_2$};
\draw[decorate,decoration={brace},thick] ([yshift=0.5em,xshift=-0.5em]d6.north west) --
([yshift=0.5em,xshift=-0.5em]d6.north west) node[midway]{${}^\text{0}$};
\draw[decorate,decoration={brace},thick] ([yshift=0mm,xshift=0mm]d6.center) --
([yshift=0mm,xshift=0mm]d6.center) node[midway]{$l_{p'}$};
\draw[decorate,decoration={brace},thick] ([yshift=0.5em,xshift=-0.5em]d7.north west) --
([yshift=0.5em,xshift=-0.5em]d7.north west) node[midway]{${}^\text{1}$};
\draw[decorate,decoration={brace},thick] ([yshift=0mm,xshift=0mm]d7.center) --
([yshift=0mm,xshift=0mm]d7.center) node[midway]{$x_4$};
\draw[decorate,decoration={brace},thick] ([yshift=0.5em,xshift=-0.5em]e1.north west) --
([yshift=0.5em,xshift=-0.5em]e1.north west) node[midway]{${}^\text{0}$};
\draw[decorate,decoration={brace},thick] ([yshift=0mm,xshift=0mm]e1.center) --
([yshift=0mm,xshift=0mm]e1.center) node[midway]{$k_1$};
\draw[decorate,decoration={brace},thick] ([yshift=0.5em,xshift=-0.5em]e2.north west) --
([yshift=0.5em,xshift=-0.5em]e2.north west) node[midway]{${}^\text{1}$};
\draw[decorate,decoration={brace},thick] ([yshift=0mm,xshift=0mm]e2.center) --
([yshift=0mm,xshift=0mm]e2.center) node[midway]{$k_2$};
\draw[decorate,decoration={brace},thick] ([yshift=0.5em,xshift=-0.5em]e3.north west) --
([yshift=0.5em,xshift=-0.5em]e3.north west) node[midway]{${}^\text{0}$};
\draw[decorate,decoration={brace},thick] ([yshift=0mm,xshift=0mm]e3.center) --
([yshift=0mm,xshift=0mm]e3.center) node[midway]{$k_{m'}$};
\draw[decorate,decoration={brace},thick] ([yshift=0.5em,xshift=-0.5em]e4.north west) --
([yshift=0.5em,xshift=-0.5em]e4.north west) node[midway]{${}^\text{1}$};
\draw[decorate,decoration={brace},thick] ([yshift=0mm,xshift=0mm]e4.center) --
([yshift=0mm,xshift=0mm]e4.center) node[midway]{$l_2$};
\draw[decorate,decoration={brace},thick] ([yshift=0.5em,xshift=-0.5em]e5.north west) --
([yshift=0.5em,xshift=-0.5em]e5.north west) node[midway]{${}^\text{0}$};
\draw[decorate,decoration={brace},thick] ([yshift=0mm,xshift=0mm]e5.center) --
([yshift=0mm,xshift=0mm]e5.center) node[midway]{$l_1$};
\draw[decorate,decoration={brace},thick] ([yshift=0.5em,xshift=-0.5em]e6.north west) --
([yshift=0.5em,xshift=-0.5em]e6.north west) node[midway]{${}^\text{1}$};
\draw[decorate,decoration={brace},thick] ([yshift=0mm,xshift=0mm]e6.center) --
([yshift=0mm,xshift=0mm]e6.center) node[midway]{$x_1$};
\draw[decorate,decoration={brace},thick] ([yshift=0.5em,xshift=-0.5em]e7.north west) --
([yshift=0.5em,xshift=-0.5em]e7.north west) node[midway]{${}^\text{0}$};
\draw[decorate,decoration={brace},thick] ([yshift=0mm,xshift=0mm]e7.center) --
([yshift=0mm,xshift=0mm]e7.center) node[midway]{$l_{p'}$};
\draw[decorate,decoration={brace},thick] ([yshift=0.5em,xshift=-0.5em]f1.north west) --
([yshift=0.5em,xshift=-0.5em]f1.north west) node[midway]{${}^\text{0}$};
\draw[decorate,decoration={brace},thick] ([yshift=0mm,xshift=0mm]f1.center) --
([yshift=0mm,xshift=0mm]f1.center) node[midway]{$k_1$};
\draw[decorate,decoration={brace},thick] ([yshift=0.5em,xshift=-0.5em]f2.north west) --
([yshift=0.5em,xshift=-0.5em]f2.north west) node[midway]{${}^\text{1}$};
\draw[decorate,decoration={brace},thick] ([yshift=0mm,xshift=0mm]f2.center) --
([yshift=0mm,xshift=0mm]f2.center) node[midway]{$k_2$};
\draw[decorate,decoration={brace},thick] ([yshift=0.5em,xshift=-0.5em]f8.north west) --
([yshift=0.5em,xshift=-0.5em]f8.north west) node[midway]{${}^\text{1}$};
\draw[decorate,decoration={brace},thick] ([yshift=0mm,xshift=0mm]f8.center) --
([yshift=0mm,xshift=0mm]f8.center) node[midway]{$k_{m''}$};
\draw[decorate,decoration={brace},thick] ([yshift=0.5em,xshift=-0.5em]f4.north west) --
([yshift=0.5em,xshift=-0.5em]f4.north west) node[midway]{${}^\text{1}$};
\draw[decorate,decoration={brace},thick] ([yshift=0mm,xshift=0mm]f4.center) --
([yshift=0mm,xshift=0mm]f4.center) node[midway]{$l_2$};
\draw[decorate,decoration={brace},thick] ([yshift=0.5em,xshift=-0.5em]f5.north west) --
([yshift=0.5em,xshift=-0.5em]f5.north west) node[midway]{${}^\text{0}$};
\draw[decorate,decoration={brace},thick] ([yshift=0mm,xshift=0mm]f5.center) --
([yshift=0mm,xshift=0mm]f5.center) node[midway]{$l_1$};
\draw[decorate,decoration={brace},thick] ([yshift=0.5em,xshift=-0.5em]f6.north west) --
([yshift=0.5em,xshift=-0.5em]f6.north west) node[midway]{${}^\text{1}$};
\draw[decorate,decoration={brace},thick] ([yshift=0mm,xshift=0mm]f6.center) --
([yshift=0mm,xshift=0mm]f6.center) node[midway]{$l_{p''}$};
\draw[decorate,decoration={brace},thick] ([yshift=0.5em,xshift=-0.5em]f7.north west) --
([yshift=0.5em,xshift=-0.5em]f7.north west) node[midway]{${}^\text{0}$};
\draw[decorate,decoration={brace},thick] ([yshift=0mm,xshift=0mm]f7.center) --
([yshift=0mm,xshift=0mm]f7.center) node[midway]{$x_2$};
}
where $m'=2m-1,m''=2m-2,p'=2p-1,p''=2p-2$, $x_1=k_{2m}+l_{2p}$, $x_2=k_{2m-1}+l_{2p-1}$, $x_3=k_{2m}+l_{2p-1}$, and $x_4=k_{2m-1}+l_{2p}$.

\bigskip
The primary goals of this paper are to study the explicit relations of K--Y MZVs $\zeta(\bfk\circledast{\bfl^\star})$ and their related variants, such as $T$-variants $T(\bfk\circledast{\bfl})$. Then using these explicit relations, we establish some explicit formulas of multiple zeta (star) values and their related variants.

The remainder of this paper is organized as follows. In Section \ref{sec2}, we first establish the explicit evaluations of integrals $\int_0^1 x^{n-1}\Li_{\bfk}(x)\,dx$ and $\int_0^1 x^{2n+b}\A(\bfk;x)\,dx$ for all positive integers $n$ and $b\in\{-1,-2\}$, where $\Li_{\bfk}(x)$ is the single-variable multiple polylogarithm (see \eqref{equ:singleLi})
and $\A(\bfk;x)$ is the Kaneko--Tsumura A-function (see \eqref{equ:defnA}). Then, for all compositions $\bfk$ and $\bfl$, using these explicit formulas obtained and by considering the two kind of integrals
\[I_L(\bfk;\bfl):=\int_0^1 \frac{\Li_{\bfk}(x)\Li_{\bfl}(x)}{x}\,dx\quad\text{and}\quad I_A(\bfk;\bfl):=\int_0^1 \frac{\A(\bfk;x)\A(\bfl;x)}{x} \, dx,\]
we establish some explicit relations of $\zeta(\bfk\circledast{\bfl^\star})$ and $T(\bfk\circledast{\bfl})$. Further, we express the integrals $I_L(\bfk;\bfl)$ and $I_A(\bfk;\bfl)$ in terms of multiple integrals associated with 2-labeled posets following the idea of Yamamoto \cite{Y2014}.

In Section \ref{sec3}, we first define a variation of the classical multiple polylogarithm function with $r$-variable
$\gl_{\bfk}(x_1,x_2,\dotsc,x_r)$ (see \eqref{equ:gl}), and give the explicit evaluation of the integral
$$ \int_0^1 x^{n-1} \gl_{\bfk}(\sigma_1x,\sigma_2x,\dotsc,\sigma_rx)\,dx, \quad \sigma_j\in\{\pm 1\}.$$
Then we will consider the integral
\[I_\gl((\bfk;\bfsi),(\bfl;\bfeps)):=
\int_0^1 \frac{\gl_{\bfk_r}(\sigma_1x,\dotsc,\sigma_rx)\gl_{\bfl_s}(\varepsilon_1x,\dotsc,\varepsilon_sx)}{x}\,dx
\]
to find some explicit relations of alternating Kaneko--Yamamoto MZVs $\z((\bfk;\bfsi)\circledast(\bfl;\bfeps)^\star)$. Further, we will find some relations involving alternating MZVs. Finally, we express the integrals $I_\gl((\bfk;\bfsi),(\bfl;\bfeps))$ in terms of multiple integrals associated with 3-labeled posets.

In Section \ref{sec4}, we define the multiple $t$-harmonic (star) sums and the function $t(\bfk;x)$ related to multiple $t$-values. Further, we establish some relations involving multiple $t$-star values.

\section{Formulas of Kaneko--Yamamoto MZVs and $T$-Variants}\label{sec2}
In this section we will prove several explicit formulas of Kaneko--Yamamoto MZVs and $T$-variants, and find some explicit relations among MZ(S)Vs and MTVs.

\subsection{Some relations of Kaneko--Yamamoto MZVs}
\begin{thm}\label{Thm1} Let $r,n\in \N$ and ${\bfk}_r:=(k_1,\dotsc,k_r)\in \N^r$. Then
\begin{align}\label{a1}
\int_0^1 x^{n-1}\Li_{{\bfk}_r}(x)dx&=\sum_{j=1}^{k_r-1} \frac{(-1)^{j-1}}{n^j}\z\left({\bfk}_{r-1},k_r+1-j\right)+\frac{(-1)^{|{\bfk}_r|-r}}{n^{k_r}}\gz^\star_n\left(1,{\bfk}_{r-1}\right)\nonumber\\
&\quad+\sum_{l=1}^{r-1} (-1)^{|{\bfk}_r^l|-l} \sum_{j=1}^{k_{r-l}-1}\frac{(-1)^{j-1}}{n^{k_r}} \gz^\star_n\left(j,{\bfk}_{r-1}^{l-1}\right)\z\left({\bfk}_{r-l-1},k_{r-l}+1-j\right),
\end{align}
where ${\Li}_{{{k_1},\dotsc,{k_r}}}(z)$ is the single-variable multiple polylogarithm function defined by
\begin{align}\label{equ:singleLi}
&{\Li}_{{{k_1},\dotsc,{k_r}}}(z): = \sum\limits_{0< {n_1} <  \cdots  < {n_r}} {\frac{{{z^{{n_r}}}}}{{n_1^{{k_1}}\cdots n_r^{{k_r}}}}},\quad z \in \left[ { - 1,1} \right).
\end{align}
\end{thm}
\begin{proof}
It's well known that multiple polylogarithms can be expressed by the iterated integral
\begin{align*}
\Li_{k_1,\dotsc,k_r}(x)=\int_0^x \frac{dt}{1-t}\left(\frac{dt}{t}\right)^{k_1-1} \dotsm\frac{dt}{1-t}\left(\frac{dt}{t}\right)^{k_r-1},
\end{align*}
where for 1-forms $\ga_1(t)=f_1(t)\, dt,\dotsc,\ga_\ell(t)=f_\ell(t)\, dt$, we define iteratively
\begin{equation*}
 \int_a^b  \ga_1(t) \cdots \ga_\ell(t) = \int_a^b \left(\int_a^y \ga_1(t)\cdots\ga_{\ell-1}(t)\right) f_\ell(y)\, dy.
\end{equation*}
Using integration by parts, we deduce the recurrence relation
\[
\int_0^1 x^{n-1}\Li_{{\bfk}_r}(x)dx=\sum_{j=1}^{k_r-1} \frac{(-1)^{j-1}}{n^j}\z({\bfk}_{r-1},k_r+1-j)+\frac{(-1)^{k_r-1}}{n^{k_r}}\sum_{j=1}^n \int_0^1 x^{j-1}\Li_{{\bfk}_{r-1}}(x)dx.
\]
Thus, we arrive at the desired formula by a direct calculation.
\end{proof}

For any string $\{s_1,\dots,s_d\}$ and $r\in \N$, we denote by $\{s_1,\dots,s_d\}_r$
the concatenated string obtained by repeating $\{s_1,\dots,s_d\}$ exactly $r$ times.
\begin{cor}\label{cor-I2}\emph{(cf. \cite{Xu2017})} For positive integers $n$ and $r$,
\begin{align*}
\int_0^1 x^{n-1}\log^r(1-x)dx=(-1)^rr!\frac{\gz^\star_n(\{1\}_{r})}{n}.
\end{align*}
\end{cor}

For any nontrivial compositions $\bfk$ and $\bfl$,
we consider the integral
\[
I_L(\bfk;\bfl):=\int_0^1 \frac{\Li_{\bfk}(x)\Li_{\bfl}(x)}{x}\,dx
\]
and use \eqref{a1} to find some explicit relations of K--Y MZVs. We prove the following theorem.

\begin{thm}\label{thm-KY} For compositions ${\bfk}_r=(k_1,\dotsc,k_r)\in \N^r$ and ${\bfl}_s=(l_1,l_2,\dotsc,l_s)\in \N^s$,
\begin{align}\label{a2}
&\sum_{j=1}^{k_r-1} (-1)^{j-1}\z\left({\bfk}_{r-1},k_r+1-j \right)\z\left({\bfl}_{s-1},l_s+j\right)+(-1)^{|{\bfk}_r|-r}\z\left({\bfl}_s\circledast\Big(1,{\bfk}_r\Big)^\star\right)\nonumber\\
&+\sum_{i=1}^{r-1} (-1)^{|{\bfk}_r^i|-i}\sum_{j=1}^{k_{r-i}-1}(-1)^{j-1} \z\left({\bfk}_{r-i-1},k_{r-i}+1-j\right)\z\left({\bfl}_s\circledast\Big(j,{\bfk}_r^i\Big)^\star\right)\nonumber\\
&=\sum_{j=1}^{l_s-1} (-1)^{j-1}\z\left({\bfl}_{s-1},l_s+1-j \right)\z\left({\bfk}_{r-1},k_r+j\right)+(-1)^{|{\bfl}_s|-s}\z\left({\bfk}_r\circledast\Big(1,{\bfl}_s\Big)^\star\right)\nonumber\\
&\quad+\sum_{i=1}^{s-1} (-1)^{|{\bfl}_s^i|-i}\sum_{j=1}^{l_{s-i}-1}(-1)^{j-1} \z\left({\bfl}_{s-i-1},l_{s-i}+1-j\right)\z\left({\bfk}_r\circledast\Big(j,{\bfl}_s^i\Big)^\star\right).
\end{align}
\end{thm}
\begin{proof}
According to the definition of multiple polylogarithm, we have
\begin{align*}
\int_0^1 \frac{\Li_{\bfk_r}(x)\Li_{\bfl_s}(x)}{x}&= \su \frac{\gz_{n-1}(\bfl_{s-1})}{n^{l_s}} \int_0^1 x^{n-1} \Li_{\bfk_r}(x)dx\\
&= \su \frac{\gz_{n-1}(\bfk_{r-1})}{n^{k_r}} \int_0^1 x^{n-1} \Li_{{\bfl}_s}(x)dx
\end{align*}
Then using \eqref{a1} with a direct calculation, we may deduce the desired evaluation.
\end{proof}

The formula in Theorem \ref{thm-KY} seems to be related to the harmonic product of Schur MZVs of anti-hook type
in \cite[Theorem 3.2]{MatsumotoNakasuji2020} and the general harmonic product
formula in \cite[Lemma 2.2]{BachmannYamasaki2018}. However, it does not seem to follow
from them easily.

As a special case, setting $r=2,s=1$ in \eqref{a2} and noting the fact that
\[\z(l_1\circledast(1,k_1,k_2)^\star)=\gz^\star(1,k_1,k_2+l_1)\]
and \[\z(l_1\circledast(j,k_2)^\star)=\gz^\star(j,l_1+k_2)\]
we find that
\begin{align}\label{a3}
&\sum_{j=1}^{k_2-1}(-1)^{j-1} \z(k_1,k_2+1-j)\z(l_1+j)+(-1)^{k_1+k_2}\gz^\star (1,k_1,k_2+l_1)\nonumber\\
&+(-1)^{k_2-1}\sum_{j=1}^{k_2-1} (-1)^{j-1} \z(k_1+1-j)\gz^\star(j,l_1+k_2)\nonumber\\
&=\sum_{j=1}^{l_1-1}(-1)^{j-1}\z(l_1+1-j)\z(k_1,k_2+j)+(-1)^{l_1-1}\z((k_1,k_2)\circledast(1,l_1)^\star).
\end{align}
On the other hand, from the definition of K--Y MZVs, it is easy to find that
\[\z((k_1,k_2)\circledast(1,l_1)^\star)=\gz^\star(k_1,1,k_2+l_1)+\gz^\star(1,k_1,k_2+l_1)-\gz^\star(k_1+1,k_2+l_1)-\gz^\star(1,k_1+k_2+l_1).\]
Hence, we can get the following corollary.
\begin{cor} For positive integers $k_1,k_2$ and $l_1$,
\begin{align}\label{a4}
&((-1)^{l_1-1}+(-1)^{k_1+k_2-1}) \gz^\star(1,k_1,k_2+l_1)+(-1)^{l_1-1}\gz^\star(k_1,1,k_2+l_1)\nonumber\\
&=\sum_{j=1}^{k_2-1}(-1)^{j-1} \z(k_1,k_2+1-j)\z(l_1+j)-(-1)^{k_2}\sum_{j=1}^{k_2-1} (-1)^{j-1} \z(k_1+1-j)\gz^\star(j,l_1+k_2)\nonumber\\
&\quad-\sum_{j=1}^{l_1-1}(-1)^{j-1}\z(l_1+1-j)\z(k_1,k_2+j)+(-1)^{l_1-1}\gz^\star(k_1+1,k_2+l_1)\nonumber\\&\quad+(-1)^{l_1-1}\gz^\star(1,k_1+k_2+l_1).
\end{align}
\end{cor}

Next, we establish an identity involving
 \emph{Arakawa--Kaneko zeta function} (see \cite{AM1999}) which is defined by
\begin{align}
\xi(k_1,\dotsc,k_r;s):=\frac{1}{\Gamma(s)} \int\limits_{0}^\infty \frac{t^{s-1}}{e^t-1}\oldLi_{k_1,\dotsc,k_r}(1-e^{-t})dt\quad (\Re(s)>0).
\end{align}
Setting variables $1-e^{-t}=x$ and $s=p+1\in \N$, we deduce
\begin{align*}
\xi(k_1,\dotsc,k_r;p+1)&=\frac{(-1)^{p}}{p!}\int\limits_{0}^1 \frac{\log^{p}(1-x){\mathrm{Li}}_{{{k_1},{k_2}, \cdots ,{k_r}}}\left( x \right)}{x}dx\\
&=\sum\limits_{n=1}^\infty \frac{\gz_{n-1}(k_1,\dotsc,k_{r-1})\gz^\star_n(\{1\}_{p})}{n^{k_r+1}}=\zeta({\bfk}\circledast (\{1\}_{p+1})^\star),
\end{align*}
where we have used Corollary \ref{cor-I2}. Clearly, the Arakawa--Kaneko zeta value is a special case of integral $I_L(\bfk;\bfl)$.
Further, setting $l_1=l_2=\cdots=l_s=1$ in Theorem \ref{thm-KY} yields
\begin{align*}
&\xi(k_1,\dotsc,k_r;s+1)=\zeta({\bfk}\circledast (\{1\}_{s+1})^\star)\\
&=\sum_{j=1}^{k_r-1} (-1)^{j-1}\z\left({\bfk}_{r-1},k_r+1-j \right)\z\left(\{1\}_{s-1},1+j\right)+(-1)^{|{\bfk}_r|-r}\z\left(\{1\}_s\circledast\Big(1,{\bfk}\Big)^\star\right)\nonumber\\
&\quad+\sum_{i=1}^{r-1} (-1)^{|{\bfk}_r^i|-i}\sum_{j=1}^{k_{r-i}-1}(-1)^{j-1} \z\left({\bfk}_{r-i-1},k_{r-i}+1-j\right)\z\left(\{1\}_s\circledast\Big(j,{\bfk}_r^i\Big)^\star\right).
\end{align*}

We end this section by the following theorem and corollary.
\begin{thm} For any positive integer $m$ and composition $\bfk=(k_1,\dotsc,k_r)$,
\begin{align}\label{czt}
&2\sum_{j=0}^{m-1} {\bar \zeta}(2m-1-2j) \su \frac{\gz_{n-1}(\bfk_{r-1})T_n(\{1\}_{2j+1})}{n^{k_r+1}}+\su \frac{\gz_{n-1}(\bfk_{r-1})S_n(\{1\}_{2m})}{n^{k_r+1}}\nonumber\\
&=\sum_{j=1}^{k_r-1} (-1)^{j-1}2^j \z(\bfk_{r-1},k_r+1-j)T(\{1\}_{2m-1},j+1)+(-1)^{|\bfk|-r}\su \frac{T_n(\{1\}_{2m-1})\gz^\star_n(1,\bfk_{r-1})}{n^{k_r+1}}\nonumber\\
&\quad+\sum_{l=1}^{r-1} (-1)^{|\bfk_r^l|-l}\sum_{j=1}^{k_{r-l}-1}(-1)^{j-1} \z(\bfk_{r-l-1},k_{r-l}+1-j)\su \frac{T_n(\{1\}_{2m-1})\gz^\star_n\Big(j,\bfk_{r-1}^{l-1}\Big)}{n^{k_r+1}}.
\end{align}
\end{thm}
\begin{proof}
On the one hand, in \cite[Theorem 3.6]{XZ2020}, we proved that
\begin{align*}
\int_{0}^1 \frac{1}{x}\cdot \oldLi_{\bfk}(x^2)\log^{2m}\xx\, dx=\frac{(2m)!}{2}\times[\text{The left-hand side of \eqref{czt}}].
\end{align*}
On the other hand, we note that
\begin{align*}
\int_{0}^1  \frac{1}{x}\cdot \oldLi_{\bfk}(x^2)\log^{2m}\xx\, dx&=(2m)!\su \frac{T_n(\{1\}_{2m-1})}{n} \int_0^1 x^{2n-1} \Li_{\bfk}(x^2)dx\\
&=(2m)!\su \frac{T_n(\{1\}_{2m-1})}{2n} \int_0^1 x^{n-1} \Li_{\bfk}(x)dx.
\end{align*}
Then using (\ref{a1}) with an elementary calculation, we have
\begin{align*}
\int_{0}^1  \frac{1}{x}\cdot \oldLi_{\bfk}(x^2)\log^{2m}\xx\, dx=\frac{(2m)!}{2}\times[\text{The right-hand side of \eqref{czt}}].
\end{align*}
Thus,  formula \eqref{czt} holds.
\end{proof}

In particular, setting $\bfk=(\{1\}_{r-1},k)$ we obtain \cite[Theorem 3.9]{XZ2020}. Setting $\bfk=(\{2\}_{r-1},k)$ we get the following corollary.

\begin{cor} For any positive integers $k,m$ and $r$,
\begin{multline}  \label{cztb}
 2\sum_{j=0}^{m-1} {\bar \zeta}(2m-1-2j) \su \frac{\gz_{n-1}(\{2\}_{r-1})T_n(\{1\}_{2j+1})}{n^{k+1}}+\su \frac{\gz_{n-1}(\{2\}_{r-1})S_n(\{1\}_{2m})}{n^{k+1}} \\
 =\sum_{j=1}^{k-1} (-1)^{j-1}2^j \z(\{2\}_{r-1},k+1-j)T(\{1\}_{2m-1},j+1) \\
 +\sum_{l=1}^{r} (-1)^{l+k} \z(\{2\}_{r-l})\su \frac{T_n(\{1\}_{2m-1})\gz^\star_n(j,\{2\}_{l-1})}{n^{k+1}}.
\end{multline}

\end{cor}

\subsection{Some relations of $T$-varinat of Kaneko--Yamamoto MZVs}
Recall that the Kaneko--Tsumura A-function ${\A}(k_1,\dotsc,k_r;z)$  (see \cite{KanekoTs2018b}) is defined by
\begin{align}\label{equ:defnA}
&{\A}(k_1,\dotsc,k_r;z): = 2^r\sum\limits_{1 \le {n_1} <  \cdots  < {n_r}\atop n_i\equiv i\ {\rm mod}\ 2} {\frac{{{z^{{n_r}}}}}{{n_1^{{k_1}}  \cdots n_r^{{k_r}}}}},\quad z \in \left[ { - 1,1} \right).
\end{align}
In this subsection, we present a series of results concerning this function.
\begin{thm} For positive integers $m$ and $n$,
\begin{align}
&\int_0^1 x^{2n-2} \A(\bfk_{2m};x)\, dx=\sum_{j=1}^{\bfk_{2m}-1} \frac{(-1)^{j-1}}{(2n-1)^j} T(\bfk_{2m-1},k_{2m}+1-j)+\frac{(-1)^{|\bfk_{2m}|}}{(2n-1)^{k_{2m}}}T_n(1,\bfk_{2m-1})\nonumber\\
&\quad+\frac{1}{(2n-1)^{k_{2m}}}\sum_{i=1}^{m-1} (-1)^{|\bfk_{2m}^{2i}|} \sum_{j=1}^{k_{2m-2i}-1} (-1)^{j-1} T(\bfk_{2m-2i-1},k_{2m-2i}+1-j)T_n(j,\bfk_{2m-1}^{2i-1})\nonumber\\
&\quad-\frac{1}{(2n-1)^{k_{2m}}}\sum_{i=0}^{m-1} (-1)^{|\bfk_{2m}^{2i+1}|} \sum_{j=1}^{k_{2m-2i-1}-1} (-1)^{j-1} T(\bfk_{2m-2i-2},k_{2m-2i-1}+1-j)S_n(j,\bfk_{2m-1}^{2i})\nonumber\\
&\quad -\frac{1}{(2n-1)^{k_{2m}}} \sum_{i=0}^{m-1} (-1)^{|\bfk_{2m}^{2i+1}|} \left(\int_0^1 \A(\bfk_{2m-2i-1},1;x)dx\right) T_n(\bfk_{2m-1}^{2i}),\label{a5}\\
&\int_0^1 x^{2n-1} \A(\bfk_{2m+1};x)\, dx=\sum_{j=1}^{\bfk_{2m+1}-1} \frac{(-1)^{j-1}}{(2n)^j} T(\bfk_{2m},k_{2m+1}+1-j)-\frac{(-1)^{|\bfk_{2m+1}|}}{(2n)^{k_{2m+1}}}T_n(1,\bfk_{2m})\nonumber\\
&\quad-\frac{1}{(2n)^{k_{2m+1}}}\sum_{i=0}^{m-1} (-1)^{|\bfk_{2m+1}^{2i+1}|} \sum_{j=1}^{k_{2m-2i}-1} (-1)^{j-1} T(\bfk_{2m-2i-1},k_{2m-2i}+1-j)T_n(j,\bfk_{2m}^{2i})\nonumber\\
&\quad+\frac{1}{(2n)^{k_{2m+1}}}\sum_{i=0}^{m-1} (-1)^{|\bfk_{2m+1}^{2i+2}|} \sum_{j=1}^{k_{2m-2i-1}-1} (-1)^{j-1} T(\bfk_{2m-2i-2},k_{2m-2i-1}+1-j)S_n(j,\bfk_{2m}^{2i+1})\nonumber\\
&\quad +\frac{1}{(2n)^{k_{2m+1}}} \sum_{i=0}^{m-1} (-1)^{|\bfk_{2m+1}^{2i+2}|} \left(\int_0^1 \A(\bfk_{2m-2i-1},1;x)dx\right) T_n(\bfk_{2m}^{2i+1}),\label{a6}\\
&\int_0^1 x^{2n-2} \A(\bfk_{2m+1};x)\,dx=\sum_{j=1}^{\bfk_{2m+1}-1} \frac{(-1)^{j-1}}{(2n-1)^j} T(\bfk_{2m},k_{2m+1}+1-j)-\frac{(-1)^{|\bfk_{2m+1}|}}{(2n-1)^{k_{2m+1}}}S_n(1,\bfk_{2m})\nonumber\\
&\quad+\frac{1}{(2n-1)^{k_{2m+1}}}\sum_{i=1}^{m} (-1)^{|\bfk_{2m+1}^{2i}|} \sum_{j=1}^{k_{2m+1-2i}-1} (-1)^{j-1} T(\bfk_{2m-2i},k_{2m+1-2i}+1-j)T_n(j,\bfk_{2m}^{2i-1})\nonumber\\
&\quad-\frac{1}{(2n-1)^{k_{2m+1}}}\sum_{i=0}^{m-1} (-1)^{|\bfk_{2m+1}^{2i+1}|} \sum_{j=1}^{k_{2m-2i}-1} (-1)^{j-1} T(\bfk_{2m-2i-1},k_{2m-2i}+1-j)S_n(j,\bfk_{2m}^{2i})\nonumber\\
&\quad -\frac{1}{(2n-1)^{k_{2m+1}}} \sum_{i=0}^{m} (-1)^{|\bfk_{2m+1}^{2i+1}|} \left(\int_0^1 \A(\bfk_{2m-2i},1;x)dx\right) T_n(\bfk_{2m}^{2i}),\label{a7}\\
&\int_0^1 x^{2n-1} \A(\bfk_{2m};x)\, dx=\sum_{j=1}^{\bfk_{2m}-1} \frac{(-1)^{j-1}}{(2n)^j} T(\bfk_{2m-1},k_{2m}+1-j)+\frac{(-1)^{|\bfk_{2m}|}}{(2n)^{k_{2m}}}S_n(1,\bfk_{2m-1})\nonumber\\
&\quad-\frac{1}{(2n)^{k_{2m}}}\sum_{i=1}^{m} (-1)^{|\bfk_{2m}^{2i-1}|} \sum_{j=1}^{k_{2m+1-2i}-1} (-1)^{j-1} T(\bfk_{2m-2i},k_{2m+1-2i}+1-j)T_n(j,\bfk_{2m-1}^{2i-2})\nonumber\\
&\quad+\frac{1}{(2n)^{k_{2m}}}\sum_{i=1}^{m-1} (-1)^{|\bfk_{2m}^{2i}|} \sum_{j=1}^{k_{2m-2i}-1} (-1)^{j-1} T(\bfk_{2m-2i-1},k_{2m-2i}+1-j)S_n(j,\bfk_{2m-1}^{2i-1})\nonumber\\
&\quad +\frac{1}{(2n)^{k_{2m}}} \sum_{i=1}^{m} (-1)^{|\bfk_{2m}^{2i}|} \left(\int_0^1 \A(\bfk_{2m-2i},1;x)dx\right) T_n(\bfk_{2m-1}^{2i-1}),\label{a8}
\end{align}
where we allow $m=0$ in \eqref{a6} and \eqref{a7}.
\end{thm}
\begin{proof}
It is easy to see that the A-function can be expressed by an iterated integral:
\begin{align*}
\A(k_1,\dotsc,k_r;x)=\int_0^x \frac{2dt}{1-t^2}\left(\frac{dt}{t}\right)^{k_1-1}
\cdots\frac{2dt}{1-t^2}\left(\frac{dt}{t}\right)^{k_r-1}.
\end{align*}
Using integration by parts, we deduce the recurrence relation
\begin{align*}
\int_0^1 x^{2n-2} \A(\bfk_r;x)\, dx&=\sum_{j=1}^{k_r-1}\frac{(-1)^{j-1}}{(2n-1)^j} T(\bfk_{r-1},k_r+1-j)+\frac{(-1)^{k_r-1}}{(2n-1)^{k_r}} \int_0^1 \A(\bfk_{r-1},1;x)\, dx\\
&\quad+\frac{(-1)^{k_r-1}}{(2n-1)^{k_r}}2\sum_{k=1}^{n-1} \int_0^1 x^{2k-1} \A(\bfk_{r-1};x)\, dx,
\end{align*}
and
\begin{align*}
\int_0^1 x^{2n-1} \A(\bfk_r;x) \, dx&=\sum_{j=0}^{k_r-2}\frac{(-1)^{j}}{(2n)^{j+1}} T(\bfk_{r-1},k_r-j)+\frac{(-1)^{k_r-1}}{(2n)^{k_r}}2\sum_{k=1}^{n} \int_0^1 x^{2k-2} \A(\bfk_{r-1};x)\, dx.
\end{align*}
Hence, using the recurrence formulas above, we may deduce the four desired evaluations
after an elementary but rather tedious computation, which we leave to the interested reader.
\end{proof}

\begin{lem}\label{equ:Aones}
For any positive integer $r$ we have
\begin{equation*}
\int_0^1  \A(\{1\}_{r};x) \, dx = -2^{1-r} \zeta(\bar r)=
\left\{
  \begin{array}{ll}
\phantom{\frac12}     \log 2, & \hbox{if $r=1$;} \\
2^{1-r}(1-2^{1-r}) \zeta(r), \qquad \ & \hbox{if $r\ge 2$.}
  \end{array}
\right.
\end{equation*}
\end{lem}

\begin{proof}
Consider the generating function
\begin{equation*}
G(u):=1+\sum_{r=1}^\infty  \left(\int_0^1  \A(\{1\}_{r};x) \, dx \right) (-2u)^r.
\end{equation*}
By definition
\begin{align*}
G(u) =\, &  1+\sum_{r=1}^\infty  (-2u)^r   \int_0^1  \int_0^{x} \left(\frac{dt}{1-t^2} \right)^r \, dx \\\
=\, &  1+\sum_{r=1}^\infty  \frac{(-2u)^r}{r!}   \int_0^1  \left( \int_0^{x} \frac{dt}{1-t^2} \right)^r \, dx \\
=\, &  1+\int_0^1 \left( \sum_{r=1}^\infty  \frac{1}{r!}  \left(-u\log \left(\frac{1+x}{1-x}\right)\right)^r \right) \, dx   \\
=\, &  \int_0^1  \left(\frac{1-x}{1+x}\right)^{u} \, dx .
\end{align*}
Taking $a=u,b=1,c=u+2$ and $t=-1$ in the formula
\begin{equation*}
{}_2F_1\left(\left.{
a,b \atop c}\right|t \right)=\frac{\Gamma(c)}{\Gamma(b)\Gamma(c-b)}
\int_0^1 v^{b-1} (1-v)^{c-b-1} (1-vt)^{-a} \,dv,
\end{equation*}
we obtain
\begin{align*}
G(u)=\, &\frac{1}{u+1} \sum_{k\ge 0} \frac{u(u+1)}{(u+k)(u+k+1)} (-1)^k \\
=\, &u\sum_{k\ge0} (-1)^k \left(\frac{1}{u+k}-\frac{1}{u+k+1}\right)\\
=\, & 1+\sum_{k\ge 1} 2 (-1)^k  \frac{u}{u+k} \\
=\, & 1-\sum_{k\ge 1} 2 (-1)^k  \sum_{r\ge 0} \left(\frac{-u}{k}\right)^{r+1} \\
=\, & 1-2  \sum_{r\ge 1} \zeta(\bar r)(-u)^r.
\end{align*}
The lemma follows immediately.
\end{proof}

\begin{thm}\label{thm-IA}
For composition $\bfk=(k_1,\dotsc,k_r)$, the integral
\[\int_0^1 \A(k_1,\dotsc,k_r,1;x)dx\]
can be expressed as a $\Q$-linear combination of alternating MZVs.
\end{thm}
\begin{proof}
It suffices to prove the integral can be expressed in terms of $\log 2$ and MMVs since these values
generate the same $\Q$-vector space as that by alternating MZVs as shown in \cite{XZ2020}.
Suppose $k_r>1$. Then
\begin{align*}
\,& \int_0^1 \A(k_1,\dotsc,k_r,1;x)\, dx\\
=\,& 2^{r+1}\sum_{\substack{ 0<n_1<\cdots<n_r<n_{r+1} \\ n_i\equiv i \pmod{2} }}
\frac{1}{n_1^{k_1}\cdots n_r^{k_r}} \left(\frac1{n_{r+1}}-    \frac1{n_{r+1}+1} \right) \\
=\,&
\left\{
  \begin{array}{ll}
  M_*(\breve{k_1},k_2,\breve{k_3},\dotsc, \breve{k_r},1) - M_*(\breve{k_1},k_2,\dotsc, \breve{k_r},\breve{1} ),& \qquad \hbox{if $2\nmid r$;} \\
 M_*(\breve{k_1},k_2,\breve{k_3},\dotsc, k_r,\breve{1}) - M_*(\breve{k_1},k_2,\dotsc,  k_r,1), & \qquad \hbox{if $2\mid r$,}
  \end{array}
\right. \\
=\,& \left\{
  \begin{array}{ll}
M(\breve{k_1},k_2,\breve{k_3},\dotsc, \breve{k_r})\big(M_*(1)-M_*(\breve{1})\big)  & \pmod{MMV}, \qquad \hbox{if $2\nmid r$;} \\
M(\breve{k_1},k_2,\breve{k_3},\dotsc, k_r)\big(M_*(\breve{1})-M_*(1)\big) & \pmod{MMV},\qquad  \hbox{if $2\mid r$,}
  \end{array}
\right. \\
=\,& \left\{
  \begin{array}{ll}
-2M(\breve{k_1},k_2,\breve{k_3},\dotsc, \breve{k_r})\log 2 & \pmod{MMV} ,\qquad  \hbox{if $2\nmid r$;} \\
2M(\breve{k_1},k_2,\breve{k_3},\dotsc, k_r)\log 2  &  \pmod{MMV},\qquad \hbox{if $2\mid r$.}
  \end{array}
\right.
\end{align*}
which can be expressed as a $\Q$-linear combination of MMVs by \cite[Theorem 7.1]{XZ2020}.

In general, we may assume $k_r>1$ and consider $\int_0^1 \A(k_1,\dotsc,k_r,\{1\}_\ell;x)\, dx$.
By induction on $\ell$, we see that
\begin{align*}
&\, \int_0^1 \A(k_1,\dotsc,k_r,\{1\}_\ell;x)\, dx\\
=&\, \left\{
  \begin{array}{ll}
M(\breve{k_1},k_2,\breve{k_3},\dotsc, \breve{k_r})\big(M_*(\bfu,1)-M_*(\bfu,\breve{1})\big)  & \pmod{MMV} , \qquad \hbox{if $2\nmid r$, $2\nmid \ell$;} \\
M(\breve{k_1},k_2,\breve{k_3},\dotsc, \breve{k_r})\big(M_*(\bfu',1,\breve{1})-M_*(\bfu',1,1)\big)  &\pmod{MMV} , \qquad  \hbox{if $2\nmid r$, $2\mid\ell$;} \\
M(\breve{k_1},k_2,\breve{k_3},\dotsc, k_r)\big(M_*(\bfv,\breve{1})-M_*(\bfv,1)\big)   & \pmod{MMV},\qquad \hbox{if $2\mid r$, $2\nmid\ell$;}\\
M(\breve{k_1},k_2,\breve{k_3},\dotsc, k_r)\big(M_*(\bfv',\breve{1},1)-M_*(\bfv',\breve{1},\breve{1})\big)  & \pmod{MMV},  \qquad\hbox{if $2\mid r$, $2\mid\ell$,}
  \end{array}
\right.
\end{align*}
where $\bfu=\{1,\breve{1}\}_{(\ell-1)/2},\bfu'=\{1,\breve{1}\}_{(\ell-2)/2}, \bfv=\{\breve{1},1\}_{(\ell-1)/2}, \bfv'=\{\breve{1},1\}_{(\ell-2)/2}.$ By Lemma \ref{equ:Aones} we see that
$M_*(\cdots,1)-M_*(\cdots,\breve{1})=\mp 2\zeta(\bar\ell)$. This finishes the proof of the theorem.
\end{proof}

\begin{exa}\label{exa-A} Applying the idea in the proof of Theorem~\ref{thm-IA} we can find that for any positive integer $k$,
\begin{align*}
\int_0^1 \A(k,1;x)\,dx&=M_*(\breve{k},1)-M_*(\breve{k},\breve{1})\\
&=M(\breve{k})(M_*(1)-M_*(\breve{1}))+M(\breve{1},\breve{k})-M(1,\breve{k})+2M\big((k+1)\breve{\, }\big).
\end{align*}
Observing that $M_*(\breve{1})-M_*(1)=2\log(2),\ M(\breve{k})=T(k),\ M(\breve{1},\breve{k})=4t(1,k)$ and $M(1,\breve{k})=S(1,k)$, we obtain
\begin{align*}
\int_0^1 \A(k,1;x)\,dx=-2\log(2)T(k)+2T(k+1)+4t(1,k)-S(1,k).
\end{align*}
\end{exa}

{}From Lemma \ref{equ:Aones} we can get the following corollary, which was proved in \cite{XZ2020}.
\begin{cor}\label{cor-II}\emph{(\cite[Theorem 3.1]{XZ2020})} For positive integers $m$ and $n$, the following identities hold.
\begin{align}
&\begin{aligned}
\int_{0}^1 t^{2n-2} \log^{2m}\tt dt&= \frac{2(2m)!}{2n-1} \sum_{j=0}^m {\bar \zeta}(2j)T_n(\{1\}_{2m-2j}),\label{ee}
\end{aligned}\\
&\begin{aligned}
\int_{0}^1 t^{2n-2} \log^{2m-1}\tt dt&= -\frac{(2m-1)!}{2n-1} \left(2\sum_{j=1}^{m} {\bar \zeta}(2j-1)T_n(\{1\}_{2m-2j}) + S_n(\{1\}_{2m-1}) \right),\label{eo}
\end{aligned}\\
&\begin{aligned}
\int_{0}^1 t^{2n-1} \log^{2m}\tt dt&=\frac{(2m)!}{n}  \left(\sum_{j=1}^{m} {\bar \zeta}(2j-1)T_n(\{1\}_{2m-2j+1})+ S_n(\{1\}_{2m})\right),\label{oe}
\end{aligned}\\
&\begin{aligned}
\int_{0}^1 t^{2n-1} \log^{2m-1}\tt dt&= -\frac{(2m-1)!}{n} \sum_{j=0}^{m-1} {\bar \zeta}(2j-2)T_n(\{1\}_{2m-2j-1}),\label{oo}
\end{aligned}
\end{align}
where ${\bar \zeta}(m):=-\zeta(\overline{ m})$, and ${\bar \zeta}(0)$ should be interpreted as $1/2$ wherever it occurs.
\end{cor}

We now derive some explicit relations about $T$-variant of K--Y MZV $T(\bfk\circledast\bfl)$ by considering the integral
\[
I_A(\bfk;\bfl):=\int_0^1 \frac{\A(\bfk;x)\A(\bfl;x)}{x} \, dx.
\]

\begin{thm}  \label{thm:S2Ts}
For positive integers $k$ and $l$, we have
\begin{multline*}
((-1)^l-(-1)^k)S(1,k+l)
=\sum_{j=1}^l (-1)^{j-1} T(l+1-j)T(k+j)+\sum_{j=1}^k (-1)^{j} T(k+1-j)T(l+j),
\end{multline*}
where $T(1):=2\log(2)$.
\end{thm}
\begin{proof} One may deduce the formula by a straightforward calculation of the integral
\begin{align*}
\int_0^1 \frac{\A(k;x)\A(l;x)}{x}\, dx.
\end{align*}
We leave the details to the interested reader.
\end{proof}

For example, setting $k=1$ and $l=2p\ (p\in\N)$ in Theorem \ref{thm:S2Ts} yields
\begin{align*}
S(1,2p+1)=\sum_{j=0}^{p-1} (-1)^{j-1} T(2p+1-j)T(j+1)-\frac{(-1)^p}{2}T^2(p+1).
\end{align*}

\begin{thm}\label{thm-TT2} For positive integers $k_1,k_2$ and $l$,
\begin{align}\label{b17}
&(-1)^{l-1}T((k_1,k_2)\circledast(1,l))+(-1)^{k_1+k_2-1}T(1,k_1,k_2+l)\nonumber\\
&=\sum_{j=1}^{k_2-1} (-1)^{j-1}T(k_1,k_2+1-j)T(l+j)-\sum_{j=1}^{l-1} (-1)^{j-1} T(l+1-j)T(k_1,k_2+j)\nonumber\\
&\quad-(-1)^{k_2}\sum_{j=1}^{k_1-1}(-1)^{j-1} T(k_1+1-j)S(j,k_2+l)-(-1)^{k_2}T(k_2+l)\int_0^1 \A(k_1,1;x) \, dx,
\end{align}
where $\int_0^1 \A(k,1;x) \, dx$ is given by Example \ref{exa-A}.
\end{thm}
\begin{proof}
From \eqref{a5} and \eqref{a6}, we deduce
\begin{align*}
\int_0^1 x^{2n-1}\A(k;x)\,dx=\sum_{j=1}^{k-1} \frac{(-1)^{j-1}}{(2n)^j}T(k+1-j)+\frac{(-1)^{k-1}}{(2n)^k}T_n(1)
\end{align*}
and
\begin{multline*}
\int_0^1 x^{2n-2}\A(k_1,k_2;x)\,dx=\sum_{j=1}^{k_2-1} \frac{(-1)^{j-1}}{(2n-1)^j}T(k_1,k_2+1-j)+\frac{(-1)^{k_1+k_2}}{(2n-1)^{k_2}}T_n(1,k_1)\\
+\frac{(-1)^{k_2-1}}{(2n-1)^{k_2}}\sum_{j=1}^{k_1-1} (-1)^{j-1} T(k_1+1-j)S_n(j)
+\frac{(-1)^{k_2-1}}{(2n-1)^{k_2}}\int_0^1 \A(k_1,1,;x)\, dx.
\end{multline*}
According to the definitions of A-functions, MTVs and MSVs, on the one hand, we have
\begin{align*}
&\int_0^1 \frac{\A(k_1,k_2;x)\A(l;x)}{x}\, dx=2\su\frac{1}{(2n-1)^l} \int_0^1 x^{2n-2}\A(k_1,k_2;x)\,dx\\
&=\sum_{j=1}^{k_2-1} (-1)^{j-1}T(k_1,k_2+1-j)T(l+j)-(-1)^{k_2}\sum_{j=1}^{k_1-1}(-1)^{j-1} T(k_1+1-j)S(j,k_2+l)\\
&\quad-(-1)^{k_2}T(k_2+l)\int_0^1 \A(k_1,1;x) \, dx+(-1)^{k_1+k_2}T(1,k_1,k_2+l).
\end{align*}
On the other hand,
\begin{multline*}
\int_0^1 \frac{\A(k_1,k_2;x)\A(l;x)}{x}\, dx=2\su\frac{T_n(k_1)}{(2n)^{k_2}} \int_0^1 x^{2n-1}\A(l;x)\,dx\\
=\sum_{j=1}^{l-1} (-1)^{j-1} T(l+1-j)T(k_1,k_2+j)+(-1)^{l-1} T((k_1,k_2)\circledast(1,l)).
\end{multline*}
Hence, combining two identities above, we obtain the desired evaluation.
\end{proof}

\begin{thm}\label{thm-TT3} For positive integers $k_1,k_2$ and $l_1,l_2$, we have
\begin{align*}
&(-1)^{k_1+k_2}T((l_1,l_2)\circledast(1,k_1,k_2)) -(-1)^{l_1+l_2}T((k_1,k_2)\circledast(1,l_1,l_2))\\
&=\sum_{j=1}^{k_2-1} (-1)^{j} T(k_1,k_2+1-j)T(l_1,l_2+j)-\sum_{j=1}^{l_2-1} (-1)^{j} T(l_1,l_2+1-j)T(k_1,k_2+j) \\
&\quad-(-1)^{k_2}\sum_{j=1}^{k_1} (-1)^{j} T(k_1+1-j)T((l_1,l_2)\circledast(j,k_2)) \\
&\quad+(-1)^{l_2}\sum_{j=1}^{l_1} (-1)^{j} T(l_1+1-j)T((k_1,k_2)\circledast(j,l_2)),
\end{align*}
where $T(1):=2\log(2).$
\end{thm}
\begin{proof}
Consider the integral
\[\int_0^1 \frac{\A(k_1,k_2;x)\A(l_1,l_2;x)}{x}\, dx.\]
By a similar argument used in the proof of Theorem \ref{thm-TT2}, we can prove Theorem \ref{thm-TT3}.
\end{proof}

Moreover, according to the definitions of Kaneko--Tsumura $\psi$-function and Kaneko--Tsumura A-function (which is a single-variable multiple polylogarithm function of level two) \cite{KanekoTs2018b,KanekoTs2019},
\begin{align}\label{a14}
\psi(k_1,\dotsc,k_r;s):=\frac{1}{\Gamma(s)} \int\limits_{0}^\infty \frac{t^{s-1}}{\sinh(t)}{\rm A}({k_1,\dotsc,k_r};\tanh(t/2))dt\quad (\Re(s)>0)
\end{align}
and
\begin{align}\label{a15}
&{\rm A}(k_1,\dotsc,k_r;z): = 2^r\sum\limits_{1 \le {n_1} <  \cdots  < {n_r}\atop n_i\equiv i\ {\rm mod}\ 2} {\frac{{{z^{{n_r}}}}}{{n_1^{{k_1}}  \cdots n_r^{{k_r}}}}},\quad z \in \left[ { - 1,1} \right).
\end{align}
Setting $\tanh(t/2)= x$ and $s =p+1\in\N$, we have
\begin{align}\label{cc8}
\psi(k_1,\dotsc,k_r;p+1)&=\frac{(-1)^{p}}{p!}\int\limits_{0}^1 \frac{\log^{p}\left(\frac{1-x}{1+x}\right)\A(k_1,\dotsc,k_r;x)}{x}dx\nonumber\\
&=\int\limits_{0}^1 \frac{\A(\{1\}_p;x)\A(k_1,\dotsc,k_r;x)}{x}dx,
\end{align}
where we have used the relation
\begin{align*}
{\rm A}({\{1\}_r};x)=\frac{1}{r!}({\rm A}(1;x))^r=\frac{(-1)^r}{ r!}\log^r\left(\frac{1-x}{1+x}\right).
\end{align*}

We remark that the Kaneko--Tsumura $\psi$-values can be regarded as a special case of the integral $I_A(\bfk;\bfl)$. So, one can prove \cite[Theorem 3.3]{XZ2020} by considering the integrals $I_A(\bfk;\bfl)$.

\subsection{Multiple integrals associated with 2-labeled posets}
According to iterated integral expressions, we know that $\Li_{\bfk(x)}$ and $\A(\bfk;x)$ satisfy the shuffle product relation. In this subsection, we will express integrals $I_L(\bfk;\bfl)$ and $I_A(\bfk;\bfl)$ in terms of multiple integral associated with 2-labeled posets, which implies that the integrals  $I_L(\bfk;\bfl)$ and $I_A(\bfk;\bfl)$ can be expressed in terms of linear combination of MZVs (or MTVs). The key properties of these integrals was first studied by Yamamoto in \cite{Y2014}.

\begin{defn}
A \emph{$2$-poset} is a pair $(X,\delta_X)$, where $X=(X,\leq)$ is
a finite partially ordered set and $\delta_X$ is a map from $X$ to $\{0,1\}$.
We often omit  $\delta_X$ and simply say ``a 2-poset $X$''.
The $\delta_X$ is called the \emph{label map} of $X$.

A 2-poset $(X,\delta_X)$ is called \emph{admissible} if
$\delta_X(x)=0$ for all maximal elements $x\in X$ and
$\delta_X(x)=1$ for all minimal elements $x\in X$.
\end{defn}

\begin{defn}
For an admissible 2-poset $X$, we define the associated integral
\begin{equation}\label{4.1}
I_j(X)=\int_{\Delta_X}\prod_{x\in X}\om^{(j)}_{\delta_X(x)}(t_x), \qquad j=1,2,
\end{equation}
where
\[\Delta_X=\bigl\{(t_x)_x\in [0,1]^X \bigm| t_x<t_y \text{ if } x<y\bigr\}\]
and
\[\om^{(1)}_0(t)=\om^{(2)}_0(t)=\frac{dt}{t}, \quad \om^{(1)}_1(t)=\frac{dt}{1-t}, \quad \om^{(2)}_1(t)=\frac{2dt}{1-t^2}. \]
\end{defn}

For the empty 2-poset, denoted $\emptyset$, we put $I_j(\emptyset):=1\ (j=1,2)$.

\begin{pro}\label{prop:shuffl2poset}
For non-comparable elements $a$ and $b$ of a $2$-poset $X$, $X^b_a$ denotes the $2$-poset that is obtained from $X$ by adjoining the relation $a<b$. If $X$ is an admissible $2$-poset, then the $2$-poset $X^b_a$ and $X^a_b$ are admissible and
\begin{equation}\label{4.2}
I_j(X)=I_j(X^b_a)+I_j(X^a_b)\quad (j=1,2).
\end{equation}
\end{pro}

Note that the admissibility of a 2-poset corresponds to
the convergence of the associated integral. We use Hasse diagrams to indicate 2-posets, with vertices $\circ$ and $\bullet$ corresponding to $\delta(x)=0$ and $\delta(x)=1$, respectively.  For example, the diagram
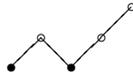
\[\begin{xy}
{(0,-4) \ar @{{*}-o} (4,0)},
{(4,0) \ar @{-{*}} (8,-4)},
{(8,-4) \ar @{-o} (12,0)},
{(12,0) \ar @{-o} (16,4)}
\end{xy} \]
represents the 2-poset $X=\{x_1,x_2,x_3,x_4,x_5\}$ with order
$x_1<x_2>x_3<x_4<x_5$ and label
$(\delta_X(x_1),\dotsc,\delta_X(x_5))=(1,0,1,0,0)$.
This 2-poset is admissible.
To describe the corresponding diagram, we introduce an abbreviation:
For a sequence $\bfk_r=(k_1,\dotsc,k_r)$ of positive integers,
we write
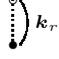
\[\begin{xy}
{(0,-3) \ar @{{*}.o} (0,3)},
{(1,-3) \ar @/_1mm/ @{-} _{\bfk_r} (1,3)}
\end{xy}\]
for the vertical diagram
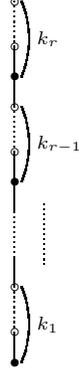
\[\begin{xy}
{(0,-24) \ar @{{*}-o} (0,-20)},
{(0,-20) \ar @{.o} (0,-14)},
{(0,-14) \ar @{-} (0,-10)},
{(0,-10) \ar @{.} (0,-4)},
{(0,-4) \ar @{-{*}} (0,0)},
{(0,0) \ar @{-o} (0,4)},
{(0,4) \ar @{.o} (0,10)},
{(0,10) \ar @{-{*}} (0,14)},
{(0,14) \ar @{-o} (0,18)},
{(0,18) \ar @{.o} (0,24)},
{(1,-24) \ar @/_1mm/ @{-} _{k_1} (1,-14)},
{(4,-3) \ar @{.} (4,-11)},
{(1,0) \ar @/_1mm/ @{-} _{k_{r-1}} (1,10)},
{(1,14) \ar @/_1mm/ @{-} _{k_r} (1,24)}
\end{xy}.\]
Hence, for admissible composition $\bfk$, using this notation of multiple associated integral, one can verify that

\begin{equation*}
\z(\bfk)=I_1\left(\ \begin{xy}
{(0,-3) \ar @{{*}.o} (0,3)},
{(1,-3) \ar @/_1mm/ @{-} _\bfk (1,3)}
\end{xy}\right)\quad\text{and}\quad T(\bfk)=I_2\left(\ \begin{xy}
{(0,-3) \ar @{{*}.o} (0,3)},
{(1,-3) \ar @/_1mm/ @{-} _\bfk (1,3)}
\end{xy}\right).
\end{equation*}

Therefore, according to the definitions of $I_L(\bfk;\bfl)$ and $I_A(\bfk;\bfl)$, and using this notation of multiple associated integral, we can get the following theorem.
\begin{thm}\label{thm-ILA} For compositions $\bfk$ and $\bfl$, we have
\begin{equation*}
I_L(\bfk;\bfl)=I_1\left(\xybox{
{(0,-9) \ar @{{*}-o} (0,-4)},
{(0,-4) \ar @{.o} (0,4)},
{(0,4) \ar @{-o} (5,9)},
{(10,-9) \ar @{{*}-o} (10,-4)},
{(10,-4) \ar @{.o} (10,4)},
{(10,4) \ar @{-} (5,9)},
{(-1,-9) \ar @/^1mm/ @{-} ^\bfk (-1,4)},
{(11,-9) \ar @/_1mm/ @{-} _{\bfl} (11,4)},
}\ \right)\quad\text{\rm and}\quad I_A(\bfk;\bfl)=I_2\left(\xybox{
{(0,-9) \ar @{{*}-o} (0,-4)},
{(0,-4) \ar @{.o} (0,4)},
{(0,4) \ar @{-o} (5,9)},
{(10,-9) \ar @{{*}-o} (10,-4)},
{(10,-4) \ar @{.o} (10,4)},
{(10,4) \ar @{-} (5,9)},
{(-1,-9) \ar @/^1mm/ @{-} ^\bfk (-1,4)},
{(11,-9) \ar @/_1mm/ @{-} _{\bfl} (11,4)},
}\ \right).
\end{equation*}
\end{thm}
\begin{proof}This follows immediately from the definitions of $I_L(\bfk;\bfl)$ and $I_A(\bfk;\bfl)$.
We leave the detail to the interested reader.
\end{proof}

It is clear that using Theorem \ref{thm-ILA}, the integrals $I_L(\bfk;\bfl)$ (or $I_A(\bfk;\bfl)$) can be expressed in terms of MZVs (or MTVs). In particular, for any positive integer $s$ the integrals $I_L(\bfk;\{1\}_s)$ and $I_A(\bfk;\{1\}_s)$ become the Arakawa--Kaneko zeta values and Kankeo--Tsumura $\psi$-values, respectively. Moreover, Kawasaki--Ohno \cite{KO2018} and Xu--Zhao \cite{XZ2020} have used the multiple integrals associated with 2-posets to prove explicit formulas for all Arakawa--Kaneko zeta values and Kankeo--Tsumura $\psi$-values.

Now, we end this section by the following duality relations. For any $n\in\N$ and composition $\bfk=(k_1,\dotsc,k_r)$, set
\begin{equation*}
    \bfk_{+n}:=(k_1,\dotsc,k_{r-1},k_r+n).
\end{equation*}

\begin{thm}\label{thmDFILA} For any $p\in\N$ and compositions of positive integers $\bfk$, $\bfl$, we have
\begin{equation*}
I_L(\bfk_{+(p-1)};\bfl)+(-1)^p I_L(\bfk;\bfl_{+(p-1)})
=\sum_{j=1}^{p-1} (-1)^{j-1} \z(\bfk_{+(p-j)})\z(\bfl_{+j})
\end{equation*}
and
\begin{equation*}
I_A(\bfk_{+(p-1)};\bfl)+(-1)^p I_A(\bfk;\bfl_{+(p-1)})
=\sum_{j=1}^{p-1} (-1)^{j-1} T(\bfk_{+(p-j)})T(\bfl_{+j}).
\end{equation*}
\end{thm}
\begin{proof}
This follows easily from the definitions of $I_L(\bfk;\bfl)$ and $I_A(\bfk;\bfl)$ by using integration by parts.
We leave the detail to the interested reader.
\end{proof}

Setting $\bfk=(\{1\}_r)$ and $\bfl=(\{1\}_s)$ in Theorem \ref{thmDFILA} and noting the duality relations $\z(\{1\}_{r-1},s+1)=\z(\{1\}_{s-1},r+1)$ and $T(\{1\}_{r-1},s+1)=T(\{1\}_{s-1},r+1)$ , we obtain the following two well-known duality formulas for  Arakawa--Kaneko zeta values and Kankeo--Tsumura $\psi$-values (see \cite{AM1999,KanekoTs2018b})
\begin{align*}
&\xi(\{1\}_{r-1},p;s+1)+(-1)^p\xi(\{1\}_{s-1},p;r+1)=\sum\limits_{j=0}^{p-2} (-1)^j \z(\{1\}_{r-1},p-j) \z(\{1\}_j,s+1)
\end{align*}
and
\begin{align*}
&\psi(\{1\}_{r-1},p;s+1)+(-1)^p\psi(\{1\}_{s-1},p;r+1)=\sum\limits_{j=0}^{p-2} (-1)^j T(\{1\}_{r-1},p-j) T(\{1\}_j,s+1).
\end{align*}

\section{Alternating variant of Kaneko--Yamamoto MZVs}\label{sec3}

\subsection{Integrals of multiple polylogarithm function with $r$-variable}
For any composition $\bfk_r=(k_1,\dotsc,k_r)\in\N^r$, we define the \emph{classical multiple polylogarithm function} with $r$-variable by
\begin{align*}
\oldLi_{\bfk_r}(x_1,\dotsc,x_r):=\sum_{0<n_1<n_2<\dotsb<n_r} \frac{x_1^{n_1}\dotsm x_r^{n_r}}{n_1^{k_1}\dotsm n_r^{k_r}}
\end{align*}
which converges if $|x_j\cdots x_r|<1$ for all $j=1,\dotsc,r$. It can be analytically continued to a multi-valued meromorphic function on $\CC^r$ (see \cite{Zhao2007d}). We also consider the following two variants. The first one is
the star version:
\begin{align*} 
\Li^\star_{\bfk_r}(x_1,\dotsc,x_r):=\sum_{0<n_1\leq n_2\leq \dotsb\leq n_r} \frac{x_1^{n_1}\dotsm x_r^{n_r}}{n_1^{k_1}\dotsm n_r^{k_r}}.
\end{align*}
The second is the most useful when we need to apply the technique of iterated integrals:
\begin{align}
\gl_{\bfk_r}(x_1,\dotsc,x_{r-1},x_r):=&\, \oldLi_{\bfk_r}(x_1/x_2,\dotsc,x_{r-1}/x_r,x_r) \notag\\
=&\,\sum_{0<n_1<n_2<\dotsb<n_r} \frac{(x_1/x_2)^{n_1}\dotsm (x_{r-1}/x_r)^{n_{r-1}}x_r^{n_r}}{n_1^{k_1}\dotsm n_{r-1}^{k_{r-1}}n_r^{k_r}}\label{equ:gl}
\end{align}
which converges if $|x_j|<1$ for all $j=1,\dotsc,r$. Namely,
\begin{equation}\label{equ:glInteratedInt}
\gl_{\bfk_r}(x_1,\dotsc,x_r)= \int_0^1 \left(\frac{x_1\, dt}{1-x_1t}\right)\left(\frac{dt}{t}\right)^{k_1-1}\cdots
\left(\frac{x_r\, dt}{1-x_r t}\right)\left(\frac{dt}{t}\right)^{k_r-1}.
\end{equation}

Similarly, we define the parametric multiple harmonic sums and parametric multiple harmonic star sums with $r$-variable are defined by
\begin{align*}
\gz_n(k_1,\dotsc,k_r;x_1,\dotsc,x_r):=\sum\limits_{0<m_1<\cdots<m_r\leq n } \frac{x_1^{m_1}\cdots x_r^{m_r}}{m_1^{k_1}\cdots m_r^{k_r}}
\end{align*}
and
\begin{align*}
\gz^\star_n(k_1,\dotsc,k_r;x_1,\dotsc,x_r):=\sum\limits_{0<m_1\leq \cdots\leq m_r\leq n} \frac{x_1^{m_1}\cdots x_r^{m_r}}{m_1^{k_1}\cdots m_r^{k_r}},
\end{align*}
respectively. Obviously,
\begin{align*}
\lim_{n\rightarrow \infty} \gz_n(k_1,\dotsc,k_r;x_1,\dotsc,x_r)=\oldLi_{k_1,\dotsc,k_r}(x_1,\dotsc,x_r)
\end{align*}
and
\begin{align*}
\lim_{n\rightarrow \infty} \gz^\star_n(k_1,\dotsc,k_r;x_1,\dotsc,x_r)=\Li^\star_{k_1,\dotsc,k_r}(x_1,\dotsc,x_r).
\end{align*}
\begin{defn}
For any two compositions of positive integers $\bfk=(k_1,\dotsc,k_r)$, $\bfl=(l_1,\dotsc,l_s)$,  $\bfsi:=(\sigma_1,\dotsc,\sigma_r)\in\{\pm 1\}^r$ and $\bfeps:=(\varepsilon_1,\dotsc,\varepsilon_s)\in\{\pm 1\}^s$, define
\begin{align}
&\z((\bfk;\bfsi)\circledast(\bfl;\bfeps)^\star)\equiv\z((k_1,\dotsc,k_r;\sigma_1,\dotsc,\sigma_r)\circledast (l_1,\dotsc,l_s;\varepsilon_1,\dotsc,\varepsilon_s)^\star)\nonumber\\
&:=\su \frac{\gz_{n-1}(k_1,\dotsc,k_{r-1};\sigma_1,\dotsc,\sigma_{r-1})
\gz^\star_n(l_1,\dotsc,l_{s-1};\varepsilon_1,\dotsc,\varepsilon_{r-1})}{n^{k_r+l_s}}(\sigma_r\varepsilon_s)^n.
\end{align}
We call them \emph{alternating Kaneko--Yamamoto MZVs}.
\end{defn}

\begin{thm}
For $n\in\N$, $\bfk_r=(k_1,\dotsc,k_r)\in\N^r$ and $\bfsi_r:=(\sigma_1,\dotsc,\sigma_r)\in\{\pm 1\}^r$, we have
\begin{align*}
&\int_0^1 x^{n-1} \gl_{k_1,\dotsc,k_r}(\sigma_1x,\dotsc,\sigma_rx)\,dx\nonumber\\
&=\sum_{j=1}^{k_r-1} \frac{(-1)^{j-1}}{n^{j-1}} \gl_{\bfk_{r-1},k_r+1-j}(\bfsi_r)
+\frac{(-1)^{k_r}}{n^{k_r}}(\sigma_r^n-1)\gl_{\bfk_{r-1},1}(\bfsi_r)\nonumber\\
&\quad-\frac{\sigma_r^n}{n^{k_r}} \sum_{l=1}^{r-1} (-1)^{|\bfk_r^l|}\sum_{j=1}^{k_{r-l}-1}(-1)^{j}\gz^\star_n\Big(j,\bfk_{r-1}^{l-1};\sigma_{r-l+1},(\bfsi_{r}\bfsi_{r-1})^{l-1}\Big) \gl_{\bfk_{r-l-1},k_{r-l}+1-j}(\bfsi_{r-l})\nonumber\\
&\quad-\frac{\sigma_r^n}{n^{k_r}} \sum_{l=1}^{r-1} (-1)^{|\bfk_{r}^{l+1}|-l}\gl_{\bfk_{r-l-1},1}(\bfsi_{r-l})\Big(\gz^\star_n\big(\bfk_{r-1}^l;\sigma_{r-l+1},(\bfsi_{r}\bfsi_{r-1})^{l-1}\big)-\gz^\star_n\big(\bfk_{r-1}^l;(\bfsi_{r}\bfsi_{r-1})^{l}\big) \Big)\nonumber\\
&\quad +(-1)^{|\bfk|-r}\frac{\sigma_r^n}{n^{k_r}} \gz^\star_n\Big(1,\bfk_{r-1};\sigma_1,(\bfsi_r\bfsi_{r-1})^{r-1}\Big),
\end{align*}
where $(\bfsi_{r}\bfsi_{r-1})^{l}:=(\sigma_{r-l+1}\sigma_{r-l},\sigma_{r-l+2}\sigma_{r-l+1},\dotsc,\sigma_r\sigma_{r-1})$ and $(\bfsi_{r}\bfsi_{r-1})^{0}:=\emptyset$. If $\sigma_r=1$ then $(\sigma_r^n-1)\gl_{\bfk_{r-1},1}(\bfsi_r):=0$, and if $\sigma_{r-l}=1$ then
\[\gl_{\bfk_{r-l-1},1}(\bfsi_{r-l})\Big(\gz^\star_n\big(\bfk_{r-1}^l;\sigma_{r-l+1},(\bfsi_{r}\bfsi_{r-1})^{l-1}\big)-\gz^\star_n\big(\bfk_{r-1}^l;(\bfsi_{r}\bfsi_{r-1})^{l}\big)\Big):=0.\]
\end{thm}
\begin{proof}
According to definition,
\begin{align*}
\frac{d}{dx}\gl_{k_1,\dotsc,k_r}(\sigma_1x,\dotsc,\sigma_{r-1}x,\sigma_rx)=
\left\{
\begin{array}{ll} \frac{1}{x} \gl_{k_1,\dotsc,k_{r-1},k_r-1}(\sigma_1x,\dotsc,\sigma_{r-1}x,\sigma_rx),
   &\quad \hbox{if $k_r\geq 2$};  \\
   \frac{\sigma_r}{1-\sigma_rx}\gl_{k_1,\dotsc,k_{r-1}}(\sigma_1x,\dotsc,\sigma_{r-1}x), &\quad\hbox{if $k_r = 1$}.   \\
\end{array}
\right.
\end{align*}
Hence, using identity above we can get the following recurrence relation
\begin{align*}
&\int_0^1 x^{n-1} \gl_{k_1,\dotsc,k_r}(\sigma_1x,\sigma_2x,\dotsc,\sigma_rx)\,dx\\
&=\sum_{j=1}^{k_r-1} \frac{(-1)^{j-1}}{n^j} \gl_{k_1,\dotsc,k_{r-1},k_r+1-j}(\sigma_1,\sigma_2,\dotsc,\sigma_r)
+\frac{(-1)^{k_r}}{n^{k_r}}(\sigma_r^n-1)\gl_{k_1,\dotsc,k_{r-1},1}(\sigma_1,\dotsc,\sigma_r)\\
&\quad-\frac{(-1)^{k_r}}{n^{k_r}}\sigma_r^n\sum_{k=1}^{n}\sigma_r^k \int_0^1 x^{k-1} \gl_{k_1,k_2,\dotsc,k_{r-1}}(\sigma_1x,\sigma_2x,\dotsc,\sigma_{r-1}x)\,dx.
\end{align*}
Thus, we obtain the desired formula by using the recurrence relation above.
\end{proof}

Letting $r=1$ and $2$, we can get the following two corollaries.
\begin{cor}\label{cor-IL} For positive integers $n,k$ and $\sigma\in\{\pm 1\}$,
\begin{align*}
\int_0^1 x^{n-1}\gl_k(\sigma x)\,dx=\frac{(-1)^{k}}{n^k}(\sigma^n-1)\gl_1(\sigma)
-(-1)^{k}\frac{\sigma^n}{n^k}\gz^\star_n(1;\sigma)-\sum_{j=1}^{k-1} \frac{(-1)^{j}}{n^j}\gl_{k+1-j}(\sigma).
\end{align*}
\end{cor}

\begin{cor}\label{cor-IIL} For positive integers $n,k_1,k_2$ and $\sigma_1,\sigma_2\in\{\pm 1\}$,
\begin{align*}
&\int_0^1 x^{n-1}\gl_{k_1,k_2}(\sigma_1 x,\sigma_2 x)\,dx\\
&=\sum_{j=1}^{k_2-1}\frac{(-1)^{j-1}}{n^j} \gl_{k_1,k_2+1-j}(\sigma_1,\sigma_2)
+(-1)^{k}\frac{\sigma_2^n}{n^{k_2}} \sum_{j=1}^{k_1-1}(-1)^{j}\gl_{k_1+1-j}(\sigma_1)\gz^\star_n(j;\sigma_2)\\
&\quad+(-1)^{k_2}\frac{\sigma_2^n-1}{n^{k_2}}\gl_{k_1,1}(\sigma_1,\sigma_2)+(-1)^{k_1+k_2}\gl_1(\sigma_1)\frac{\sigma_2^n}{n^{k_2}}\Big(\gz^\star_n(k_1;\sigma_2)-\gz^\star_n(k_1;\sigma_2\sigma_1) \Big)\\
&\quad+(-1)^{k_1+k_2}\frac{\sigma_2^n}{n^{k_2}}\gz^\star_n(1,k_1;\sigma_1,\sigma_2\sigma_1).
\end{align*}
\end{cor}

Clearly, setting $\sigma_1=\sigma_2=\cdots=\sigma_r=1$ gives the formula \eqref{a1}.

\subsection{Explicit formulas for alternating Kaneko--Yamamoto MZVs}

Obviously, we can consider the integral
\[I_\gl((\bfk;\bfsi),(\bfl;\bfeps)):=\int_0^1 \frac{\gl_{\bfk_r}(\sigma_1x,\dotsc,\sigma_rx)\gl_{\bfl_s}(\varepsilon_1x,\dotsc,\varepsilon_sx)}{x}\,dx\]
to find some explicit relations of $\z((\bfk;\bfsi)\circledast(\bfl;\bfeps)^\star)$. We have the following theorems.

\begin{thm} For positive integers $k,l$ and $\sigma,\varepsilon\in\{\pm 1\}$,
\begin{align}
&(-1)^k\Li^\star_{1,k+l}(\sigma,\sigma\varepsilon)-(-1)^l\Li^\star_{1,k+l}(\varepsilon,\sigma\varepsilon)\nonumber\\
&=\sum_{j=1}^{k-1} (-1)^{j-1} \gl_{k+1-j}(\sigma)\gl_{l+j}(\varepsilon)-\sum_{j=1}^{l-1} (-1)^{j-1} \gl_{l+1-j}(\varepsilon)\gl_{k+j}(\sigma)\nonumber\\
&\quad+(-1)^l\gl_1(\varepsilon)(\gl_{k+l}(\sigma)-\gl_{k+l}(\sigma\varepsilon))
-(-1)^k\gl_1(\sigma)(\gl_{k+l}(\varepsilon)-\gl_{k+l}(\sigma\varepsilon)),
\end{align}
where if $\sigma=1$ then $\gl_1(\sigma)(\gl_{k+l}(\varepsilon)-\gl_{k+l}(\sigma\varepsilon)):=0$. Similarly, if $\varepsilon=1$ then $\gl_1(\varepsilon)(\gl_{k+l}(\sigma)-\gl_{k+l}(\sigma\varepsilon)):=0$.
\end{thm}
\begin{proof}
Considering the integral $\int_0^1 \frac{\gl_k(\sigma x)\gl_l(\varepsilon x)}{x}\,dx$ and using Corollary \ref{cor-IL} with an elementary calculation, we obtain the formula.
\end{proof}

\begin{thm} For positive integers $k_1,k_2,l$ and $\sigma_1,\sigma_2,\varepsilon\in\{\pm 1\}$,
\begin{align}\label{c7}
&\sum_{j=1}^{l-1} (-1)^{j-1} \gl_{l+1-j}(\varepsilon)\oldLi_{k_1,k_2+j}(\sigma_1\sigma_2,\sigma_2)
 -(-1)^l \z((k_1,k_2;\sigma_1\sigma_2,\sigma_2)\circledast(1,l;\varepsilon,\varepsilon)^\star)\nonumber\\
&\quad-(-1)^l\gl_{1}(\varepsilon)\Big(\oldLi_{k_1,k_2+l}(\sigma_1\sigma_2,\sigma_2)-\oldLi_{k_1,k_2+l}(\sigma_1\sigma_2,\sigma_2\varepsilon)\Big)\nonumber\\
&=\sum_{j=1}^{k_2-1} (-1)^{j-1} \gl_{k_1,k_2+1-j}(\sigma_1,\sigma_2)\gl_{l+j}(\varepsilon)
-(-1)^{k_2}\sum_{j=1}^{k_1-1}(-1)^{j-1} \gl_{k_1+1-j}(\sigma_1)\Li^\star_{j,k_2+l}(\sigma_2,\varepsilon\sigma_2)\nonumber\\
&\quad-(-1)^{k_2}\gl_{k_1,1}(\sigma_1,\sigma_2)\Big(\gl_{k_2}(\varepsilon)-\gl_{k_2}(\varepsilon\sigma_2) \Big)
+(-1)^{k_1+k_2} \Li^\star_{1,k_1,k_2+l}(\sigma_1,\sigma_2\sigma_1,\sigma_2\varepsilon),
\nonumber\\
&\quad+(-1)^{k_1+k_2} \gl_1(\sigma_1)\Big(\Li^\star_{k_1,k_2+l}(\sigma_2,\varepsilon\sigma_2)-\Li^\star_{k_1,k_2+l}(\sigma_2\sigma_1,\varepsilon\sigma_2)  \Big)\end{align}
where if $\varepsilon=1$ then $\gl_{1}(\varepsilon)\Big(\oldLi_{k_1,k_2+l}(\sigma_1\sigma_2,\sigma_2)-\oldLi_{k_1,k_2+l}(\sigma_1\sigma_2,\sigma_2\varepsilon)\Big):=0$; if $\sigma_1=1$ then $\gl_1(\sigma_1)\Big(\Li^\star_{k_1,k_2+l}(\sigma_2,\varepsilon\sigma_2)-\Li^\star_{k_1,k_2+l}(\sigma_2\sigma_1,\varepsilon\sigma_2)  \Big):=0$ and if $\sigma_2=1$ then $\gl_{k_1,1}(\sigma_1,\sigma_2)\Big(\gl_{k_2}(\varepsilon)-\gl_{k_2}(\varepsilon\sigma_2) \Big):=0$.
\end{thm}
\begin{proof}
Similarly, considering the integral $\int_0^1 \frac{\gl_{k_1,k_2}(\sigma_1 x,\sigma_1 x)\gl_l(\varepsilon x)}{x}\,dx$ and using Corollary \ref{cor-IIL} with an elementary calculation, we prove the formula.
\end{proof}

On the other hand, according to definition, we have
\begin{align*}
&\z((k_1,k_2;\sigma_1\sigma_2,\sigma_2)\circledast(1,l;\varepsilon,\varepsilon)^\star)=\su \frac{\gz_{n-1}(k_1;\sigma_1\sigma_2)\gz^\star_n(1;\varepsilon)}{n^{k_2+l}}(\sigma_2\varepsilon)^n\\
&=\Li^\star_{k_1,1,k_2+l}(\sigma_1\sigma_2,\varepsilon,\sigma_2\varepsilon)+\Li^\star_{1,k_1,k_2+l}(\varepsilon,\sigma_1\sigma_2,\sigma_2\varepsilon)-\Li^\star_{k_1+1,k_2+l}(\sigma_1\sigma_2\varepsilon,\sigma_2\varepsilon)
-\Li^\star_{1,k_1+k_2+l}(\varepsilon,\sigma_1\varepsilon).
\end{align*}
Substituting it into \eqref{c7} yields the following corollary.
\begin{cor}\label{cor:c8}
For positive integers $k_1,k_2,l$ and $\sigma_1,\sigma_2,\varepsilon\in\{\pm 1\}$,
\begin{align*}
&(-1)^{l}\Li^\star_{k_1,1,k_2+l}(\sigma_1\sigma_2,\varepsilon,\sigma_2\varepsilon)
+(-1)^{l}\Li^\star_{1,k_1,k_2+l}(\varepsilon,\sigma_1\sigma_2,\sigma_2\varepsilon) +(-1)^{k_1+k_2} \Li^\star_{1,k_1,k_2+l}(\sigma_1,\sigma_2\sigma_1,\sigma_2\varepsilon) \\
&=\sum_{j=1}^{k_2-1} (-1)^{j} \gl_{k_1,k_2+1-j}(\sigma_1,\sigma_2)\gl_{l+j}(\varepsilon)
-(-1)^{k_2}\sum_{j=1}^{k_1-1}(-1)^{j} \gl_{k_1+1-j}(\sigma_1)\Li^\star_{j,k_2+l}(\sigma_2,\varepsilon\sigma_2) \\
&\quad-\sum_{j=1}^{l-1} (-1)^{j} \gl_{l+1-j}(\varepsilon)\oldLi_{k_1,k_2+j}(\sigma_1\sigma_2,\sigma_2)
+(-1)^{k_2}\gl_{k_1,1}(\sigma_1,\sigma_2)\Big(\gl_{k_2}(\varepsilon)-\gl_{k_2}(\varepsilon\sigma_2) \Big) \\
&\quad-(-1)^{k_1+k_2} \gl_1(\sigma_1)\Big(\Li^\star_{k_1,k_2+l}(\sigma_2,\varepsilon\sigma_2)-\Li^\star_{k_1,k_2+l}(\sigma_2\sigma_1,\varepsilon\sigma_2)  \Big) \\
&\quad-(-1)^l\gl_{1}(\varepsilon)\Big(\oldLi_{k_1,k_2+l}(\sigma_1\sigma_2,\sigma_2)-\oldLi_{k_1,k_2+l}(\sigma_1\sigma_2,\sigma_2\varepsilon)\Big) \\
&\quad+(-1)^{l}\Li^\star_{k_1+1,k_2+l}(\sigma_1\sigma_2\varepsilon,\sigma_2\varepsilon)
+(-1)^{l}\Li^\star_{1,k_1+k_2+l}(\varepsilon,\sigma_1\varepsilon).
\end{align*}
\end{cor}

Clearly, setting $\sigma_1=\sigma_2=\varepsilon=1$ in Corollary \ref{cor:c8} gives the formula \eqref{a4}. We also find numerous explicit relations involving alternating MZVs. For example, letting $k_1=k_2=l=2$ and $\sigma_1=\eps=-1, \sigma_2=1$, we have
\begin{align*}
 \gz^\star(\bar 2,\bar 1,\bar 4)+2\gz^\star(\bar 1,\bar 2,\bar 4)
&=3 \Li_4\left(\frac{1}{2}\right) \zeta (3)-\frac{7 \pi ^4 \zeta (3)}{128}+\frac{61 \pi ^2 \zeta (5)}{192}-\frac{105 \zeta (7)}{128}+\frac{1}{8} \zeta (3) \log ^4(2)\\&\quad-\frac{1}{8} \pi ^2 \zeta (3) \log ^2(2)+\frac{63}{16} \zeta (3)^2 \log (2)-\frac{61 \pi ^6 \log (2)}{10080},
\end{align*}
where we used Au's Mathematica package \cite{Au2020}.

\subsection{Multiple integrals associated with 3-labeled posets}

In this subsection, we introduce the multiple integrals associated with 3-labeled posets, and express the integrals $I_\gl((\bfk;\bfsi),(\bfl;\bfeps))$ in terms of multiple integrals associated with 3-labeled posets.

\begin{defn}
A \emph{$3$-poset} is a pair $(X,\delta_X)$, where $X=(X,\leq)$ is
a finite partially ordered set and $\delta_X$ is a map from $X$ to $\{-1,0,1\}$.
We often omit  $\delta_X$ and simply say ``a 3-poset $X$''.
The $\delta_X$ is called the \emph{label map} of $X$.

Similar to 2-poset, a 3-poset $(X,\delta_X)$ is called \emph{admissible}?
if $\delta_X(x) \ne 1$ for all maximal
elements and $\delta_X(x) \ne 0$ for all minimal elements $x \in X$.
\end{defn}

\begin{defn}
For an admissible $3$-poset $X$, we define the associated integral
\begin{equation}
I(X)=\int_{\Delta_X}\prod_{x\in X}\omega_{\delta_X(x)}(t_x),
\end{equation}
where
\[\Delta_X=\bigl\{(t_x)_x\in [0,1]^X \bigm| t_x<t_y \text{ if } x<y\bigr\}\]
and
\[\omega_{-1}(t)=\frac{dt}{1+t},\quad \omega_0(t)=\frac{dt}{t}, \quad \omega_1(t)=\frac{dt}{1-t}.\]
\end{defn}

For the empty 3-poset, denoted $\emptyset$, we put $I(\emptyset):=1$.

\begin{pro}\label{prop:shuffl3poset}
For non-comparable elements $a$ and $b$ of a $3$-poset $X$, $X^b_a$ denotes the $3$-poset that is obtained from $X$ by adjoining the relation $a<b$. If $X$ is an admissible $3$-poset, then the $3$-poset $X^b_a$ and $X^a_b$ are admissible and
\begin{equation}
I(X)=I(X^b_a)+I(X^a_b).
\end{equation}
\end{pro}

Note that the admissibility of a $3$-poset corresponds to
the convergence of the associated integral. We use Hasse diagrams to indicate $3$-posets, with vertices $\circ$ and ``$\bullet\ \sigma$" corresponding to $\delta(x)=0$ and $\delta(x)=\sigma\ (\sigma\in\{\pm 1\})$, respectively. For convenience, if $\sigma=1$, replace ``$\bullet\ 1$" by $\bullet$ and if $\sigma=-1$, replace ``$\bullet\ -1$" by ``$\bullet\ {\bar1}$".  For example, the diagram
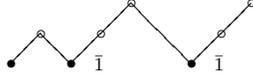
\[\begin{xy}
{(0,-4) \ar @{{*}-o} (4,0)},
{(4,0) \ar @{-{*}} (8,-4)},
{(8,-4) \ar @{-o}_{\bar 1} (12,0)},
{(12,0) \ar @{-o} (16,4)},
{(16,4) \ar @{-{*}} (24,-4)},
{(24,-4) \ar @{-o}_{\bar 1} (28,0)},
{(28,0) \ar @{-o} (32,4)}
\end{xy} \]
represents the $3$-poset $X=\{x_1,x_2,x_3,x_4,x_5,x_6,x_7,x_8\}$ with order
$x_1<x_2>x_3<x_4<x_5>x_6<x_7<x_8$ and label
$(\delta_X(x_1),\dotsc,\delta_X(x_8))=(1,0,-1,0,0,-1,0,0)$. For composition $\bfk=(k_1,\dotsc,k_r)$ and $\bfsi\in\{\pm 1\}^r$ (admissible or not),
we write
\[\begin{xy}
{(0,-3) \ar @{{*}.o} (0,3)},
{(1,-3) \ar @/_1mm/ @{-} _{(\bfk,\bfsi)} (1,3)}
\end{xy}\]
for the `totally ordered' diagram:
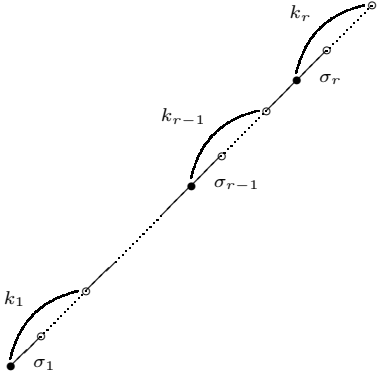
\[\begin{xy}
{(0,-24) \ar @{{*}-o}_{\sigma_1} (4,-20)},
{(4,-20) \ar @{.o} (10,-14)},
{(10,-14) \ar @{-} (14,-10)},
{(14,-10) \ar @{.} (20,-4)},
{(20,-4) \ar @{-{*}} (24,0)},
{(24,0) \ar @{-o}_{\sigma_{r-1}}(28,4)},
{(28,4) \ar @{.o} (34,10)},
{(34,10) \ar @{-{*}} (38,14)},
{(38,14) \ar @{-o}_{\sigma_r} (42,18)},
{(42,18) \ar @{.o} (48,24)},
{(0,-23) \ar @/^2mm/ @{-}^{k_1} (9,-14)},
{(24,1) \ar @/^2mm/ @{-}^{k_{r-1}} (33,10)},
{(38,15) \ar @/^2mm/ @{-}^{k_r} (47,24)}
\end{xy} \]
If $k_i=1$, we understand the notation $\begin{xy}
{(0,-5) \ar @{{*}-o}_{\sigma_i} (4,-1)},
{(4,-1) \ar @{.o} (10,5)},
{(0,-4) \ar @/^2mm/ @{-}^{k_i} (9,5)}
\end{xy}$ as a single $\bullet\ {\sigma_i}$.
We see from \eqref{equ:glInteratedInt}
\begin{align}\label{5.19}
I\left(\ \begin{xy}
{(0,-3) \ar @{{*}.o} (0,3)},
{(1,-3) \ar @/_1mm/ @{-} _{(\bfk,\bfsi)} (1,3)}
\end{xy}\right)=\frac{\gl_{k_1,\dotsc,k_r}(\sigma_1,\sigma_2,\dotsc,\sigma_r)}{\sigma_1\sigma_2\cdots \sigma_r}.
\end{align}

Therefore, according to the definition of $I_\gl((\bfk;\bfsi),(\bfl;\bfeps))$, and using this notation of multiple associated integral, we can get the following theorem.
\begin{thm}\label{thm-ILA-} For compositions $\bfk\equiv \bfk_r$ and $\bfl\equiv\bfl_s$ with $\bfsi\in\{\pm 1\}^r$ and $\bfeps\in\{\pm 1\}^s$,
\begin{equation*}
I_\gl((\bfk;\bfsi),(\bfl;\bfeps))=I\left(\xybox{
{(0,-9) \ar @{{*}-o} (0,-4)},
{(0,-4) \ar @{.o} (0,4)},
{(0,4) \ar @{-o} (5,9)},
{(10,-9) \ar @{{*}-o} (10,-4)},
{(10,-4) \ar @{.o} (10,4)},
{(10,4) \ar @{-} (5,9)},
{(-1,-9) \ar @/^1mm/ @{-} ^{(\bfk,\bfsi)} (-1,4)},
{(11,-9) \ar @/_1mm/ @{-} _{(\bfl,\bfeps))} (11,4)},
}\ \right).
\end{equation*}
\end{thm}
\begin{proof}This follows immediately from the definition of $I_\gl((\bfk;\bfsi),(\bfl;\bfeps))$.
We leave the detail to the interested reader.
\end{proof}

Finally, we end this section by the following a theorem which extends \cite[Theorem~ 4.1]{KY2018} to level two.
\begin{thm} For any non-empty compositions $\bfk_r$, $\bfl_s$ and $\bfsi_r\in\{\pm 1\}^r$, we have
\begin{align}\label{5.21}
I\left( \raisebox{16pt}{\begin{xy}
{(-3,-18) \ar @{{*}-}_{\sigma_1'} (0,-15)},
{(0,-15) \ar @{{o}.} (3,-12)},
{(3,-12) \ar @{{o}.} (9,-6)},
{(9,-6) \ar @{{*}-}_{\sigma_r'} (12,-3)},
{(12,-3) \ar @{{o}.} (15,0)},
{(15,0) \ar @{{o}-} (18,3)},
{(18,3) \ar @{{o}-} (21,6)},
{(21,6) \ar @{{o}.} (24,9)},
{(24,9) \ar @{{o}-} (27,3)},
{(27,3) \ar @{{*}-} (30,6)},
{(30,6) \ar @{{o}.} (33,9)},
{(33,9) \ar @{{o}-} (35,5)},
{(37,6) \ar @{.} (41,6)},
{(42,3) \ar @{{*}-} (45,6)},
{(45,6) \ar @{{o}.{o}} (48,9)},
{(-3,-17) \ar @/^1mm/ @{-}^{k_1} (2,-12)},
{(9,-5) \ar @/^1mm/ @{-}^{k_r} (14,0)},
{(18,4) \ar @/^1mm/ @{-}^{l_s} (23,9)},
{(28,3) \ar @/_1mm/ @{-}_{l_{s-1}} (33,8)},
{(43,3) \ar @/_1mm/ @{-}_{l_1} (48,8)},
\end{xy}}
\right)=\frac{\z((\bfk_r;\bfsi_r)\circledast(\bfl_s;\{1\}_s)^\star)}{\sigma_1'\sigma_2'\cdots\sigma_r'},
\end{align}
where $\sigma_j'=\sigma_j\sigma_{j+1}\cdots\sigma_r$, and $\bullet\ \sigma_j'$ corresponding to $\delta(x)=\sigma_j'$.
\end{thm}
\begin{proof} The proof is done straightforwardly by computing the multiple integral
on the left-hand side of \eqref{5.21} as a repeated integral ``from left to right'' using the key
ideas in the proof of \eqref{equ:glInteratedInt} and \cite[Corollary 3.1]{Y2014}.
\end{proof}

If letting all $\sigma_i=1\ (i=1,2,\dotsc,r)$, then we obtain the ``integral-series" relation of Kaneko--Yamamoto \cite{KY2018}.

From Proposition \ref{prop:shuffl3poset} and (\ref{5.19}), it is clear that the left-hand side of (\ref{5.21}) can be expressed in terms of a linear combination of alternating multiple zeta values. Hence, we can find many linear relations of alternating multiple zeta values from (\ref{5.21}). For example,
\begin{align}
&2\gl_{1,1,3}(\sigma_1',\sigma_2',1) +2\gl_{1,1,3}(\sigma_1',1,\sigma_2')+2\gl_{1,1,3}(1,\sigma_1',\sigma_2')\nonumber\\&\quad+\gl_{1,2,2}(\sigma_1',1,\sigma_2')+\gl_{1,2,2}(1,\sigma_1',\sigma_2')+\gl_{2,1,2}(1,\sigma_1',\sigma_2')\nonumber\\
&=\z(2,1,2;1,\sigma_1,\sigma_2)+\z(1,2,2;\sigma_1,1,\sigma_2)+\z(3,2;\sigma_1,\sigma_2)+\z(1,4;\sigma_1,\sigma_2).
\end{align}
If $(\sigma_1,\sigma_2)=(1,1)$ and $(1,-1)$, then we get the following two cases
\begin{align*}
&6\zeta(1,1,3)+2\zeta(1,2,2)+\zeta(2,1,2)=\zeta(1,2,2)+\zeta(2,1,2)+\zeta(3,2)+\zeta(1,4),
\end{align*}
and
\begin{align*}
&2\zeta(1,{\bar 1},3)+2\zeta({\bar 1},{\bar 1},{\bar 3})+2\zeta({\bar 1},{ 1},{\bar 3})+\zeta({\bar 1},{\bar 2},{\bar 2})+\zeta({\bar 1},{ 2},{\bar 2})+\zeta({\bar 2},1,{\bar 2})\\
&\quad=\zeta(2,1,{\bar 2})+\zeta(1,2,{\bar 2})+\zeta(3,{\bar 2})+\zeta(1,{\bar 4}).
\end{align*}

\section{Integrals about multiple $t$-harmonic (star) sums}\label{sec4}
Similar to MHSs and MHSSs, we can define the following their $t$-versions.
\begin{defn} For any $n, r\in\N$ and composition $\bfk:=(k_1,\dotsc,k_r)\in\N^r$,
\begin{align}
&t_n(k_1,\dotsc,k_r):=\sum_{0<n_1<n_2<\dotsb<n_r\leq n} \frac{1}{(2n_1-1)^{k_1}(2n_2-1)^{k_2}\dotsm (2n_r-1)^{k_r}},\label{t-1}\\
&t^\star_n(k_1,\dotsc,k_r):=\sum_{0<n_1\leq n_2\leq \dotsb\leq n_r\leq n} \frac{1}{(2n_1-1)^{k_1}(2n_2-1)^{k_2}\dotsm (2n_r-1)^{k_r}},\label{t-2}
\end{align}
where we call \eqref{t-1} and \eqref{t-2} are multiple $t$-harmonic sums and multiple $t$-harmonic star sums, respectively. If $n<r$ then ${t_n}(\bfk):=0$ and ${t_n}(\emptyset )={t^\star _n}(\emptyset ):=1$.
\end{defn}

For composition $\bfk:=(k_1,\dotsc,k_r)$, define
\begin{align*}
L(k_1,\dotsc,k_r;x):=\frac{1}{2^{k_1+\dotsb+k_r}} \Li_{k_1,\dotsc,k_r}(x^2)
\end{align*}
where $L(\emptyset;x):=1$. Set $L(k_1,\dotsc,k_r):=L(k_1,\dotsc,k_r;1)$. Similarly, define
\begin{align*}
t(k_1,\dotsc,k_r;x):&=\sum_{0<n_1<n_2<\dotsb<n_r} \frac{x^{2n_r-1}}{(2n_1-1)^{k_1}(2n_2-1)^{k_2}\dotsm (2n_r-1)^{k_r}}\\
&=\su \frac{t_{n-1}(k_1,\dotsc,k_{r-1})}{(2n-1)^{k_r}}x^{2n-1},
\end{align*}
where $t(\emptyset;x):=1/x$. Note that $t(k_1,\dotsc,k_r;1)=t(k_1,\dotsc,k_r)$.

\begin{thm} For composition $\bfk:=(k_1,\dotsc,k_r)$ and positive integer $n$,
\begin{align}\label{d3}
&\int_0^1 x^{2n-2} L(\bfk_r;x)\,dx=\sum_{j=1}^{k_r-1} \frac{(-1)^{j-1}}{(2n-1)^j}L(\bfk_{r-1},k_r+1-j)
    + \frac{(-1)^{|\bfk|-r}}{(2n-1)^{k_r}} t^\star_n(1,\bfk_{r-1})\nonumber\\
&\quad+\frac{1}{(2n-1)^{k_r}} \sum_{l=1}^{r-1} (-1)^{|\bfk_r^l|-l}\sum_{j=1}^{k_{r-l}-1} (-1)^{j-1}L(\bfk_{r-l-1},k_{r-l}+1-j)t^\star_n(j,\bfk_{r-1}^{l-1})\nonumber\\
&\quad-\frac{1}{(2n-1)^{k_r}}\sum_{l=0}^{r-1} (-1)^{|\bfk_r^{l+1}|-l-1}\left(\int_0^1 \frac{L(\bfk_{r-l-1},1;x)}{x^2}\,dx\right)t^\star_n(\bfk_{r-1}^l).
\end{align}
\end{thm}
\begin{proof}
By the simple substitution $t\to t^2/x^2$ in \eqref{equ:glInteratedInt} we see quickly that
\begin{align*}
L(k_1,\dotsc,k_r;x)=\int_{0}^x \frac{tdt}{1-t^2}\left(\frac{dt}{t}\right)^{k_1-1}\dotsm \frac{tdt}{1-t^2}\left(\frac{dt}{t}\right)^{k_r-1}.
\end{align*}
By an elementary calculation, we deduce the recurrence relation
\begin{align*}
\int_0^1 x^{2n-2} L(\bfk_r;x)\,dx&=\sum_{j=1}^{k_r-1} \frac{(-1)^{j-1}}{(2n-1)^j}L(\bfk_{r-1},k_r+1-j) -\frac{(-1)^{k_r-1}}{(2n-1)^{k_r}} \int_0^1 \frac{L(\bfk_{r-1},1;x)}{x^2}dx\\
&\quad+\frac{(-1)^{k_r-1}}{(2n-1)^{k_r}}\sum_{l=1}^n \int_0^1 x^{2l-2}L(\bfk_{r-1};x)dx.
\end{align*}
Hence, using the recurrence relation, we obtain the desired evaluation by direct calculations.
\end{proof}

\begin{thm} For composition $\bfk:=(k_1,\dotsc,k_r)$ and positive integer $n$,
\begin{align}
&\int_0^1 x^{2n-2} t(\bfk_r;x)\,dx=\sum_{j=1}^{k_r-1} \frac{(-1)^{j-1}}{(2n-1)^j}t(\bfk_{r-1},k_r+1-j) + \frac{(-1)^{|\bfk|-r}}{(2n-1)^{k_r}} s^\star_n(1,\bfk_{r-1})\nonumber\\
&\quad+\frac{1}{(2n-1)^{k_r}} \sum_{l=1}^{r-1} (-1)^{|\bfk_r^l|-l}\sum_{j=1}^{k_{r-l}-1} (-1)^{j-1}t(\bfk_{r-l-1},k_{r-l}+1-j)\widehat{t}^\star_n(j,\bfk_{r-1}^{l-1})\nonumber\\
&\quad+\frac{1}{(2n-1)^{k_r}}\sum_{l=0}^{r-1} (-1)^{|\bfk_r^{l+1}|-l-1}\left(\int_0^1t(\bfk_{r-l-1},1;x)\,dx\right)\widehat{t}^\star_n(\bfk_{r-1}^l),
\end{align}
where
\begin{align*}
&\widehat{t}^\star_n(k_1,\dotsc,k_r):=\sum_{2\leq n_1\leq n_2\leq \dotsb\leq n_r\leq n} \frac{1}{(2n_1-1)^{k_1}(2n_2-1)^{k_2}\cdots(2n_r-1)^{k_r}},\\
&s^\star_n(k_1,\dotsc,k_r):=\sum_{2\leq n_1\leq n_2\leq \dotsb\leq n_r\leq n} \frac{1}{(2n_1-2)^{k_1}(2n_2-1)^{k_2}\cdots(2n_r-1)^{k_r}}.
\end{align*}
\end{thm}
\begin{proof}
By definition we have
\begin{align*}
\frac{d}{dx}t({{k_1}, \cdots ,k_{r-1},{k_r}}; x)= \left\{ {\begin{array}{*{20}{c}} \frac{1}{x} t({{k_1}, \cdots ,{k_{r-1}},{k_r-1}};x)
   {\ \ (k_r\geq 2),}  \\
   {\frac{x}{1-x^2}t({{k_1}, \cdots ,{k_{r-1}}};x)\;\;\;\ \ \ (k_r = 1),}  \\
\end{array} } \right.
\end{align*}
where $t(\emptyset;x):=1/x$. Hence, we obtain the iterated integral
\begin{align*}
t(k_1,\dotsc,k_r;x)=\int_{0}^x \frac{dt}{1-t^2}\left(\frac{dt}{t}\right)^{k_1-1}\frac{tdt}{1-t^2}\left(\frac{dt}{t}\right)^{k_{2}-1} \cdots \frac{tdt}{1-t^2}\left(\frac{dt}{t}\right)^{k_r-1}.
\end{align*}
By an elementary calculation, we deduce the recurrence relation
\begin{align*}
\int_0^1 x^{2n-2} t(\bfk_r;x)\,dx&=\sum_{j=1}^{k_r-1} \frac{(-1)^{j-1}}{(2n-1)^j}t(\bfk_{r-1},k_r+1-j) +\frac{(-1)^{k_r-1}}{(2n-1)^{k_r}} \int_0^1 t(\bfk_{r-1},1;x)\,dx\\
&\quad+\frac{(-1)^{k_r-1}}{(2n-1)^{k_r}}\sum_{l=2}^n \int_0^1 x^{2l-2}t(\bfk_{r-1};x)dx.
\end{align*}
Hence, using the recurrence relation, we obtain the desired evaluation by direct calculations.
\end{proof}

\begin{thm}\label{thm:L1111}
For any positive integer $r$, $\int_0^1 \frac{L(\{1\}_{r};x)}{x^2} \, dx$ can be expressed as a $\Q$-linear combinations of
products of $\log 2$ and Riemann zeta values. More precisely, we have
\begin{equation*}
1-\sum_{r\ge 1} \left(\int_0^1 \frac{L(\{1\}_{r};x)}{x^2} \, dx\right) u^r
=\exp\left( \sum_{n=1}^\infty \frac{\z(\bar n)}{n}u^n\right)
=\exp\left(-\log(2)u-\sum_{n=2}^\infty \frac{1-2^{1-n}}{n}\z(n)u^n\right).
\end{equation*}
\end{thm}
\begin{proof}
Consider the generating function
\begin{equation*}
F(u):=1-\sum_{r=1}^\infty  2^r\left(\int_0^1 \frac{L(\{1\}_{r};x)}{x^2}  \, dx\right) u^r.
\end{equation*}
By definition
\begin{align*}
F(u) =\, &  1-\sum_{r=1}^\infty  u^r   \int_0^1  \int_0^{x^2} \left(\frac{dt}{1-t} \right)^r \frac{dx}{x^2} \\\
=\, &  1-\sum_{r=1}^\infty  \frac{u^r}{r!}   \int_0^1  \left( \int_0^{x^2} \frac{dt}{1-t} \right)^r \frac{dx}{x^2} \\
=\, &  1- \int_0^1 \left( \sum_{r=1}^\infty  \frac{(-u\log(1-x^2))^r}{r!} \right)  \frac{dx}{x^2}  \\
=\, &  1+\int_0^1  \Big( (1-x^2)^{-u}-1\Big)  d(x^{-1}) 
=
\frac{\Gamma(1-u) \Gamma(1/2)}{\Gamma(1/2-u)}
\end{align*}
by integration by parts followed by the substitution $x=\sqrt{t}$.
Using the expansion
\begin{align*}
\Gamma(1-u)=\exp\left(\gamma u+\sum_{n=2}^\infty \frac{\z(n)}{n}u^n\right)\qquad(|u|<1).
\end{align*}
and setting $x=1/2-u$ in the duplication formula $\Gamma(x)\Gamma(x+1/2)=2^{1-2x}\sqrt{\pi}\Gamma(2x)$, we obtain
\begin{align*}
\log\Gamma(1/2-u)=\frac{\log\pi}{2}+\gamma u+2u\log(2)+\sum_{n=2}^\infty \frac{(2^n-1)\z(n)}{n}u^n\qquad(|u|<1/2).
\end{align*}
Therefore
\begin{equation*}
F(u)=\exp\left(-2\log(2)u-\sum_{n=2}^\infty \frac{2^n-2}{n}\z(n)u^n\right)=
\exp\left( \sum_{n=1}^\infty \frac{2^n}{n}\z(\bar n)u^n\right)
\end{equation*}
by the facts that $\zeta(\bar1)=-\log 2$ and $2^n \zeta(\bar n)=(2-2^n)\zeta(n)$ for $n\ge 2$.
The theorem follows immediately.
\end{proof}

Clearly, $\int_0^1 \frac{L(\{1\}_{r};x)}{x^2}\,dx\in\Q[\log(2),\z(2),\z(3),\z(4),\ldots]$.
For example,
\begin{align} \label{equ:Lxdepth=1}
&\int_0^1 \frac{L(1;x)}{x^2}\,dx=\log(2),\\
&\int_0^1 \frac{L(1,1;x)}{x^2}\,dx=\frac{1}{4}\z(2)-\frac1{2}\log^2(2), \notag\\
&\int_0^1 \frac{L(1,1,1;x)}{x^2}\,dx=\frac1{4}\z(3)+\frac1{6}\log^3(2)-\frac1{4}\z(2)\log(2). \notag
\end{align}

More generally, using a similar argument as in the proof of Theorem \ref{thm-IA}, we can prove the following more general results.
\begin{thm}\label{thm:LandtIntegrals}
Let $r,n$ be two non-negative integers and $\bfk_r=(k_1,\dotsc,k_r)\in\N^r$ with $\bfk_0=\emptyset$. Then
one can express all of the integrals
\begin{equation*}
\int_0^1 \frac{L(k_1,\dotsc,k_r,1;x)}{x^n}\,dx\ \quad (0\le n\le 2r+2)
\end{equation*}
and
\begin{equation*}
 \int_0^1 \frac{t(k_1,\dotsc,k_r,1;x)}{x^n}\,dx\ \quad(0\le n\le 2r+1)
\end{equation*}
as $\Q$-linear combinations of alternating MZVs (and number $1$ for $\int_0^1 L(\bfk_r,1;x)\,dx$).
\end{thm}

\begin{proof} The case $n=1$ is trivial as both integrals are clearly already MMVs after the integration.

If $n=0$ then we have
\begin{equation*}
 \int_0^1 t(k_1,\dotsc,k_r,1;x) \,dx
=\int_0^1   \sum_{0<n_1<\dots<n_r<m} \frac{x^{2m-1} \,dx}{(2n_1-1)^{k_1}\cdots (2n_r-1)^{k_r}(2m-1)}
=\frac{\lim_{N\to \infty} c_N}{2^{r+1}}
\end{equation*}
where
\begin{align}
c_N =\, &\sum_{0<n_1<\dots<n_r<m\le N} \frac{2^{r+1}}{(2n_1-1)^{k_1}\cdots (2n_r-1)^{k_r}(2m-1)(2m)} \notag\\
=\,& \sum_{0<n_1<\dots<n_r<m\le 2N}  \frac{(1-(-1)^{n_1})\cdots (1-(-1)^{n_r})}{n_1^{k_1}\cdots n_r^{k_r}}
\left(\frac{1-(-1)^m}{m}-\frac{1-(-1)^m}{m+1}\right) \notag\\
=\,&- 2\sum_{0<n_1<\dots<n_r<m\le 2N}  \frac{(1-(-1)^{n_1})\cdots (1-(-1)^{n_r})(-1)^m}{n_1^{k_1}\cdots n_r^{k_r}m}\notag\\
+\,&\sum_{0<n_1<\dots<n_r\le 2N}  \frac{(1-(-1)^{n_1})\cdots (1-(-1)^{n_r})}{n_1^{k_1}\cdots n_r^{k_r}}
\sum_{m=n_r+1}^{2N}  \left(\frac{1+(-1)^m}{m}-\frac{1-(-1)^m}{m+1}\right) \notag\\
=\,&- 2\sum_{0<n_1<\dots<n_r<m\le 2N}  \frac{(1-(-1)^{n_1})\cdots (1-(-1)^{n_r})(-1)^m}{n_1^{k_1}\cdots n_r^{k_r}m}\notag\\
+\,&\sum_{0<n_1<\dots<n_r\le 2N}  \frac{(1-(-1)^{n_1})\cdots (1-(-1)^{n_r})}{n_1^{k_1}\cdots n_r^{k_r}}
  \left(\frac{1-(-1)^{n_r}}{n_r+1}\right). \label{equ:tintInductionStep}
\end{align}
By partial fraction decomposition
\begin{equation*}
\frac1{n^k(n+1)}=\sum_{j=2}^k \frac{(-1)^{k-j}}{n^j}-(-1)^k\left(\frac{1}{n}-\frac{1}{n+1}\right)
\end{equation*}
setting $n=n_r$, $k=k_r+1$ and taking $N\to\infty$ in the above, we may assume $k_r=1$ without loss of generality.
Now if $r=1$ then by \eqref{equ:tintInductionStep}
\begin{equation*}
c_N =  2\sum_{0<n<m\le 2N}\frac{(-1)^{n+m}-(-1)^m}{nm} +  \sum_{0<n\le 2N} \frac{(1-(-1)^n)^2}{n(n+1)}.
\end{equation*}
Hence
\begin{equation}\label{equ:tIntr=1}
 \int_0^1 t(1,1;x)\,dx=\frac14 \big(2\zeta(\bar1,\bar1)-2\zeta(1,\bar1)+4\log 2\big)=\log 2-\frac14 \zeta(2)
\end{equation}
by \cite[Proposition 14.2.5]{Z2016}. If $r>1$ then
\begin{align*}
\sum_{n=m}^{2N} \frac{1-(-1)^{n}}{n(n+1)}
=  \sum_{n=m}^{2N} \frac{1}{n(n+1)} - \sum_{n=m}^{2N} \frac{(-1)^{n}}{n}+\sum_{n=m}^{2N} \frac{(-1)^{n}}{n+1}
= \frac{1+(-1)^{m}}{m}  - 2\sum_{n=m}^{2N} \frac{(-1)^{n}}{n}.
\end{align*}
Taking $m=n_{r-1}+1$ and $n=n_r$ in \eqref{equ:tintInductionStep} we get
\begin{align*}
c_N =\,&- 2\sum_{0<n_1<\dots<n_r<m\le 2N}  \frac{(1-(-1)^{n_1})\cdots (1-(-1)^{n_r})(-1)^m}{n_1^{k_1}\cdots n_r^{k_r}m} \\
+\,&2\sum_{j=2}^{k_r} (-1)^{k_r-j} \sum_{0<n_1<\dots<n_r<m\le 2N}  \frac{(1-(-1)^{n_1})\cdots (1-(-1)^{n_r})}{n_1^{k_1}\cdots n_{r-1}^{k_{r-1}}n_r^{j}} \\
-\,& 4(-1)^{k_r} \sum_{0<n_1<\dots<n_r<2N}  \frac{(1-(-1)^{n_1})\cdots (1-(-1)^{n_{r-1}})}{n_1^{k_1}\cdots n_{r-1}^{k_{r-1}}}
  \left(\frac{1}{n_{r-1}+1}  - \sum_{n_r=n_{r-1}+1}^{2N} \frac{(-1)^{n_r}}{n_r} \right).
\end{align*}
Here, when $r=1$ the last line above degenerates to $4(-1)^{k_r} \sum_{n_1=1}^{2N} \frac{(-1)^{n_1}}{n_1}$.
Taking $N\to\infty$ and using induction on $r$, we see that the claim for
$\int_0^1 t(\bfk_r,1;x)\,dx$ in the theorem follows.

The computation of $\int_0^1 L(\bfk_r,1;x)\,dx$ is completely similar to that of $\int_0^1 t(\bfk_r,1;x)\,dx$. Thus we can get
\begin{align*}
\int_0^1 L(\bfk_r,1;x)\,dx=
\,& \frac{1}{2^r} \sum_{0<n_1<\dots<n_r<m}  \frac{(1+(-1)^{n_1})\cdots (1+(-1)^{n_r})(-1)^m}{n_1^{k_1}\cdots n_r^{k_r}m} \\
+\,&\frac{1}{2^r} \sum_{j=2}^{k_r} (-1)^{k_r-j} \sum_{0<n_1<\dots<n_r }  \frac{(1+(-1)^{n_1})\cdots (1+(-1)^{n_r})}{n_1^{k_1}\cdots n_{r-1}^{k_{r-1}} n_r^{j}} \\
-\,&  \frac{2(-1)^{k_r}}{2^r} \sum_{0<n_1<\dots<n_r} \frac{(1+(-1)^{n_1})\cdots (1+(-1)^{n_{r-1}})(-1)^{n_r}}{n_1^{k_1}\cdots n_{r-1}^{k_{r-1}} n_r}\\
-\,& \frac{2(-1)^{k_r}}{2^r}  \sum_{0<n_1<\dots<n_r} \frac{(1+(-1)^{n_1})\cdots (1+(-1)^{n_{r-1}})}{n_1^{k_1}\cdots n_{r-1}^{k_{r-1}}(n_{r-1}+1)}.
\end{align*}
Here, when $r=1$ the last line above degenerates to $-(-1)^{k_1}$.
So by induction on $r$ we see that the claim for $\int_0^1 L(\bfk_r,1;x)\,dx$ is true.

Similarly, if $n=2$ then we can apply the same technique as above to get
\begin{equation*}
 \int_0^1 \frac{L(k_1,\dotsc,k_r,1;x)}{x^2}\,dx
=\int_0^1 \frac{1}{2^{k_1+\dots+k_r+1}} \sum_{0<n_1<\dots<n_r<m} \frac{x^{2m-2}\,dx}{n_1^{k_1}\cdots n_r^{k_r}m}
=\frac{\lim_{N\to \infty} d_N}{2^{r+1}}
\end{equation*}
where
\begin{align*}
d_N=\,& \sum_{0<n_1<\dots<n_r<m\le N} \frac{2^{r+1}}{(2n_1)^{k_1}\cdots (2n_r)^{k_r}2m(2m-1)}    \\
=\,& \sum_{0<n_1<\dots<n_r<m\le 2N}  \frac{(1+(-1)^{n_1})\cdots (1+(-1)^{n_r})}{n_1^{k_1}\cdots n_r^{k_r}}
\left(\frac{1+(-1)^m}{m-1}-\frac{1+(-1)^m}{m}\right) \\
=\,& \sum_{0<n_1<\dots<n_r<m\le 2N}  \frac{(1+(-1)^{n_1})\cdots (1+(-1)^{n_r})}{n_1^{k_1}\cdots n_r^{k_r}}
\left[\left(\frac{1-(-1)^m}{m}-\frac{1+(-1)^m}{m}\right) \right. \\
\,& \hskip7cm  + \left.  \left(\frac{1+(-1)^m}{m-1}-\frac{1-(-1)^m}{m}\right)\right]\\
=\,&- 2\sum_{0<n_1<\dots<n_r<m\le 2N}  \frac{(1+(-1)^{n_1})\cdots (1+(-1)^{n_r})(-1)^m}{n_1^{k_1}\cdots n_r^{k_r}m}\\
+\,&\sum_{0<n_1<\dots<n_r<2N}  \frac{(1+(-1)^{n_1})\cdots (1+(-1)^{n_r})}{n_1^{k_1}\cdots n_r^{k_r}}
\sum_{m=n_r+1}^{2N}  \left(\frac{1+(-1)^m}{m-1}-\frac{1-(-1)^m}{m}\right) \\
=\,& -2\sum_{0<n_1<\dots<n_r<m\le 2N}  \frac{(1+(-1)^{n_1})\cdots (1+(-1)^{n_r})(-1)^m}{n_1^{k_1}\cdots n_r^{k_r}m}\\
+ \,&\sum_{0<n_1<\dots<n_r<2N}  \frac{(1+(-1)^{n_1})\cdots (1+(-1)^{n_r})}{n_1^{k_1}\cdots n_r^{k_r}}
 \left(\frac{1-(-1)^{n_r}}{n_r}\right) \\
=\,& -2\sum_{0<n_1<\dots<n_r<m\le 2N}  \frac{(1+(-1)^{n_1})\cdots (1+(-1)^{n_r})(-1)^m}{n_1^{k_1}\cdots n_r^{k_r}m}\\
\to \,& -2\sum_{\eps_j=\pm1,1\le j\le r} \zeta(k_1,\dotsc,k_r,1;\eps_1,\dotsc,\eps_r,-1)
\end{align*}
as $N\to\infty$. Hence,
\begin{equation} \label{equ:LintInductionStep}
 \int_0^1 \frac{L(k_1,\dotsc,k_r,1;x)}{x^2}\,dx
=\frac{-1}{2^r}\sum_{\eps_j=\pm1,1\le j\le r} \zeta(k_1,\dotsc,k_r,1;\eps_1,\dotsc,\eps_r,-1).
\end{equation}

By exactly the same approach as above, we find that
\begin{align} \notag
\int_0^1 \frac{t(k_1,\dotsc,k_r,1;x)}{x^2}\,dx
=&\,\frac{-1}{2^r}\sum_{0<n_1<\dots<n_r<m }  \frac{(1-(-1)^{n_1})\cdots (1-(-1)^{n_r})(-1)^m}{n_1^{k_1}\cdots n_r^{k_r}m}\\
=&\,\frac{-1}{2^r}\sum_{\eps_j=\pm1,1\le j\le r} \eps_1\cdots\eps_r\zeta(k_1,\dotsc,k_r,1;\eps_1,\dotsc,\eps_r,-1). \label{equ:tintoverx2}
\end{align}

More generally, for any larger values of $n$ we may use the partial fraction technique and similar argument as above to express
the integrals in Theorem~\eqref{thm:LandtIntegrals} as $\Q$-linear combinations of alternating MZVs.
So we leave the details to the interested reader.
This finishes the proof of the theorem.
\end{proof}

\begin{exa}
For positive integer $k>1$, by the computation in $n=0$ case in the proof the Theorem~\ref{thm:LandtIntegrals} we get
\begin{align*}
\int_0^1 t(k,1;x)\,dx 
=\, &  \frac12 \big(\zeta(\bar{k},\bar1)-\zeta(k,\bar1) \big) -(-1)^k \log 2+\frac12\sum_{j=2}^k (-1)^{k-j}\big(\zeta(j)-\zeta(\bar{j})\big),\\
\int_0^1 L(k,1;x)\,dx =\, &  \frac12 \big(\zeta(\bar{k},\bar1)+\zeta(k,\bar1) \big)-(-1)^k  +(-1)^k \log 2+\frac12\sum_{j=2}^k (-1)^{k-j}\big(\zeta(j)+\zeta(\bar{j})\big) \\
=\, &  \frac12 \big(\zeta(\bar{k},\bar1)+\zeta(k,\bar1) \big)-(-1)^k  +(-1)^k \log 2+\sum_{j=2}^k \frac{(-1)^{k-j}}{2^j} \zeta(j) .
\end{align*}
\end{exa}

\begin{exa}
For positive integer $k>1$, we see from \eqref{equ:LintInductionStep} and \eqref{equ:tintoverx2} that
\begin{align}\label{equ:Lr=1}
\int_0^1 \frac{L(k,1;x)}{x^2}\,dx=&\, -\frac12\big( \zeta(k,\bar1)+\zeta(\bar{k},\bar1)\big),\\
\int_0^1 \frac{t(k,1;x)}{x^2}\,dx=&\, \frac12\big( \zeta(k,\bar1)-\zeta(\bar{k},\bar1)\big). \notag
\end{align}
Taking $r=2$ in \eqref{equ:LintInductionStep} and \eqref{equ:tintoverx2} we get
\begin{align*}
\int_0^1 \frac{L(k_1,k_2,1;x)}{x^2}\,dx=-\frac14\big(\zeta(k_1,k_2,\bar1)+
\zeta(k_1,\bar{k_2},\bar1)+\zeta(\bar{k_1},k_2,\bar1)+\zeta(\bar{k_1},\bar{k_2},\bar1)\big),\\
\int_0^1 \frac{t(k_1,k_2,1;x)}{x^2}\,dx=-\frac14\big(\zeta(k_1,k_2,\bar1)-
\zeta(k_1,\bar{k_2},\bar1)-\zeta(\bar{k_1},k_2,\bar1)+\zeta(\bar{k_1},\bar{k_2},\bar1)\big).
\end{align*}
\end{exa}

\begin{re}
It is possible to give an induction proof of Theorem~\ref{thm:LandtIntegrals} using the regularized values of MMVs as
defined by \cite[Definition 3.2]{YuanZh2014b}. However,
the general formula for the integral of $L(\bfk,1;x)$ would be implicit. To illustrate the idea for computing $\int_0^1 t(\bfk,1;x)\,dx$, we consider the case $\bfk=k\in \N$.
Notice that
\begin{align*}
\sum_{0<m<n<N} \frac{1}{(2m-1)(2n-1)2n}=&\, \sum_{\substack{0<m<n<2N\\ m,n\ \text{odd}}}
\frac{1}{m}\left(\frac{1}{n}-\frac{1}{n+1} \right) \\
=&\, \sum_{\substack{0<m<n<2N\\ m,n\ \text{odd}}} \frac{1}{mn} -
\sum_{\substack{0<m<n\le 2N\\ m \ \text{odd}, n \ \text{even}}} \frac{1}{mn}+
\sum_{\substack{0<m<2N\\ m \ \text{odd} }} \frac{1}{m(m+1)}.
\end{align*}
By using regularized values, we see that
\begin{align*}
\int_0^1 t(1,1;x)\,dx =\, & \sum_{0<m<n} \frac{1}{(2m-1)(2n-1)2n}=\frac14\big(M_*(\breve{1},\breve{1})-M_*(\breve{1},1)\big)
+\log 2.
\end{align*}
We have
\begin{equation*}
M_*(\breve{1},\breve{1})=\frac12\big(M_*(\breve{1})^2-2M_*(\breve{2})\big)=\frac12\big((T+2\log 2)^2-2M(\breve{2})\big),
\end{equation*}
Since $2M(\breve{2})=3\zeta(2)$,
\begin{equation*}
M_\sha(\breve{1},\breve{1})= \frac12\rho\big((T+2\log 2)^2-3\zeta(2)\big)=\frac12\big((T+\log 2)^2-2\zeta(2)\big)
\end{equation*}
by \cite[Theorem 2.7]{Z2016}. On the other hand,
\begin{equation*}
\rho\big(M_*(\breve{1},1)\big)=M_\sha(\breve{1},1)= \frac12 M_\sha(\breve{1})^2=\frac12(T+\log 2)^2.
\end{equation*}
Since $\rho$ is an $\R$-linear map,
\begin{equation*}
\int_0^1 t(1,1;x)\,dx =\log 2+\frac14\rho\big(M_*(1,\breve{1})-M_*(1,1)\big)=\log 2-\frac14\zeta(2).
\end{equation*}
which agrees with \eqref{equ:tIntr=1}.
\end{re}

Similarly, by considering some related integrals we can establish many relations
involving multiple $t$-star values. For example, from \eqref{d3} we have
\begin{align*}
\int_0^1 x^{2n-2} L(k_1,k_2;x)dx&=\sum_{j=1}^{k_2-1} \frac{(-1)^{j-1}}{(2n-1)^j}L(k_1,k_2+1-j)+\frac{(-1)^{k_2}}{(2n-1)^{k_2}} \int_0^1 \frac{L(k_1,1;x)}{x^2}dx\\
&\quad-\frac{(-1)^{k_2}}{(2n-1)^{k_2}} \sum_{j=1}^{k_1-1} (-1)^{j-1} L(k_1+1-j)t^\star_n(j)\\
&\quad-\frac{(-1)^{k_1+k_2}}{(2n-1)^{k_2}} \log(2)t^\star_n(k_1)+\frac{(-1)^{k_1+k_2}}{(2n-1)^{k_2}}t^\star_n(1,k_1).
\end{align*}
Hence, considering the integral $\int_0^1 \frac{\A(l;x)L(k_1,k_2;x)}{x}\,dx$ or $\int_0^1 \frac{t(l;x)L(k_1,k_2;x)}{x}\,dx$, we can get the following theorem.
\begin{thm}  For positive integers $k_1,k_2$ and $l$,
\begin{align}
&\sum_{j=1}^{k_2-1} (-1)^{j-1} L(k_1,k_2+1-j)T(l+j)+(-1)^{k_2}T(k_2+l) \int_0^1 \frac{L(k_1,1;x)}{x^2}\,dx\nonumber\\
&-(-1)^{k_2} 2\sum_{j=1}^{k_1-1}(-1)^{j-1} L(k_1+1-j)t^\star(j,k_2+l)\nonumber\\
&-(-1)^{k_1+k_2}2\log(2)t^\star(k_1,k_2+l)+(-1)^{k_1+k_2}2t^\star(1,k_1,k_2+l)\nonumber\\
&=\frac{1}{2^{k_1+k_2}} \sum_{j=1}^{l-1} \frac{(-1)^{j-1}}{2^j} T(l+1-j)\z(k_1,k_2+j)-\frac{(-1)^l}{2^{k_1+k_2+l}} \su \frac{\gz_{n-1}(k_1)T_n(1)}{n^{k_2+l}},
\end{align}
where $\int_0^1 \frac{L(k_1,1;x)}{x^2}\,dx$ is given by \eqref{equ:Lr=1}.
\end{thm}

{\bf Acknowledgments.} The first author is supported by the Scientific Research
Foundation for Scholars of Anhui Normal University and the University Natural Science Research Project of Anhui Province (Grant No. KJ2020A0057).

 \end{document}